\documentclass[reqno,oneside]{amsart}
\usepackage[margin=3cm]{geometry}
\usepackage{amssymb,amsmath,amsthm}
\usepackage{amsfonts}
\usepackage{rotating}
\usepackage[utf8]{inputenc}
\usepackage[T1]{fontenc}
\usepackage{bbm}
\usepackage{mathtools}
\mathtoolsset{showonlyrefs=true}
\usepackage{graphicx}
\usepackage[boxed,oldcommands,norelsize]{algorithm2e}
\usepackage[numbers,sort&compress]{natbib}
\usepackage{multirow}
\usepackage{hhline}
\usepackage{colortbl}
\usepackage{subcaption}
\newcolumntype{L}{>{\centering\arraybackslash}m{0.46\textwidth}}
\usepackage{arydshln}
\usepackage[hidelinks]{hyperref}

\newtheorem{lemma}{Lemma}[section]
\newtheorem{theorem}{Theorem}[section]
\newtheorem{definition}{Definition}[section]
\usepackage{comment}
\usepackage[nolist]{acronym}

\newcommand{\eqarrayconf}{\arraycolsep=1.4pt\def\arraystretch{1.2}}

\newcommand{\F}{\boldsymbol{\mathsf{F}}}
\newcommand{\B}{\boldsymbol{\mathsf{B}}}
\newcommand{\K}{\boldsymbol{\mathsf{K}}}
\newcommand{\sK}{\tilde{\boldsymbol{\mathsf{K}}}}

\newcommand{\0}{\boldsymbol{\mathsf{0}}}

\newcommand{\rr}{\boldsymbol{r}}

\newcommand{\normal}{\boldsymbol{n}}
\newcommand{\x}{\boldsymbol{x}}

\newcommand{\gradient}{\boldsymbol\nabla}

\DeclareMathOperator{\spn}{span}

\newcommand{\half}{\frac{1}{2}}
\def\ij{{ij}} % Allready defined
\newcommand{\bij}{{\boldsymbol{ij}}} 

\newcommand{\bji}{{\boldsymbol{ji}}}
\def\i{{\boldsymbol{i}}}
\def\j{{\boldsymbol{j}}}
\def\ix{{\i_{\x} }}
\def\itemp{{i_t}}

\newcommand{\conv}{\boldsymbol{\beta}}

\newcommand{\diff}{\mu}
\newcommand{\force}{g}

\newcommand{\linfty}{L^\infty}
\newcommand{\ltwo}{L^2}
\newcommand{\hone}{H^1}

\newcommand{\inflowboundary}{\Gamma_{\rm in}}
\newcommand{\outflowboundary}{\Gamma_{\rm out}}
\newcommand{\stdomain}{\mathcal{Q}}
\newcommand{\domain}{\Omega}
\newcommand{\side}{L}

\newcommand{\fespace}{V_h}
\newcommand{\testspace}{V_{h0}}

\newcommand{\knots}{{\xi}}

\newcommand{\nodes}{\mathcal{N}_h}

\newcommand{\neighborhood}[1][]{\mathcal{N}_h^{#1,s}}
\newcommand{\stneighborhood}[1][]{\mathcal{N}_h^{#1}}
\newcommand{\support}[1]{\mathcal{S}_h^{#1}}
\newcommand{\sneighborhood}[1][]{\neighborhood}

\newcommand{\stsymneigh}[1][]{\mathcal{N}^{{#1},\sym}_h}

\newcommand{\unk}{u_h}

\newcommand{\uboundary}{\overline{u}}
\newcommand{\test}{v_h}
\newcommand{\shapef}[1][]{\varphi_{#1}}

\usepackage{stmaryrd}
\newcommand{\smax}{\max{}_{\sigma_h}}

\newcommand{\detector}[1][]{\alpha_{#1}}

\newcommand{\graphl}{\ell}

\newcommand{\jump}[1]{\left\llbracket #1 \right\rrbracket}
\newcommand{\mean}[1]{%
  \sbox0{%
    \mathsurround=0pt % just for safety
    $\left\{\vphantom{#1}\right.\kern-\nulldelimiterspace$%
  }%
  \sbox2{\{}%
  \ifdim\ht0=\ht2
    \{\kern-.625\wd2 \{#1\}\kern-.625\wd2 \}%
  \else
    \left\{\kern-.7\wd0\left\{#1\right\}\kern-.7\wd0\right\}%
  \fi
}

\newcommand{\smthlimit}[1]{Z\left(#1\right)}
\newcommand{\absn}[2][\varepsilon_h]{\left\vert #2 \right\vert_{1,#1}}
\newcommand{\absd}[2][\varepsilon_h]{\left\vert #2 \right\vert_{2,#1}}
\newcommand{\smthdetector}[1][]{ %
	\ifx\empty#1\empty%
	\alpha_{\varepsilon_h} %
	\else %
	\alpha_{\varepsilon_h,#1} %
	\fi}
\newcommand{\smthspacedetector}[1][]{ %
	\ifx\empty#1\empty%
	\alpha_{\varepsilon_h}^s %
	\else %
	\alpha_{\varepsilon_h,#1}^s %
	\fi}
\newcommand{\smthtimedetector}[1][]{ %
	\ifx\empty#1\empty%
	\alpha_{\varepsilon_h}^t %
	\else %
	\alpha_{\varepsilon_h,#1}^t %
	\fi}
\newcommand{\smthstdetector}[1][]{ %
	\ifx\empty#1\empty%
	\alpha_{\varepsilon_h}^{st} %
	\else %
	\alpha_{\varepsilon_h,#1}^{st} %
	\fi}

\newcommand{\sdetectorAprox}[1][]{ %
 \ifx\empty#1\empty%
    \tilde{\alpha}_{\varepsilon_h} %
 \else %
    \tilde{\alpha}_{\varepsilon_h,#1} %
 \fi}
\newcommand{\artdif}{\nu}
\newcommand{\smthartdif}{\tilde{\nu}}
\newcommand{\sym}{{\rm sym}}

\newcommand{\correction}[1]{#1}

\def\bsbeta{\boldsymbol{\beta}}

\begin{document}

\title[Maximum-principle preserving space--time isogeometric analysis]{Maximum-principle preserving space--time isogeometric analysis}

\author{ Jes\'us Bonilla$^{\dag\ddag}$  \and Santiago Badia$^{\dag\ddag}$ }
\thanks{
	$^\dag$ CIMNE -- Centre Internacional de M\`etodes Num\`erics en Enginyeria, Esteve Terradas 5, 08860 Castelldefels,
	Spain (\{jbonilla,sbadia\}@cimne.upc.edu). \\ 
	\indent$^\ddag$
	Department of Civil and Environmental Engineering. Universitat Polit\`ecnica de Catalunya, Jordi Girona 1-3, Edifici C1, 08034
	Barcelona, Spain.
}

\renewcommand{\thefootnote}{\arabic{footnote}}

\begin{abstract} 
	In this work we propose a nonlinear stabilization technique for convection--diffusion--reaction and pure transport problems discretized with  space--time isogeometric analysis. The stabilization is based on a graph-theoretic artificial diffusion operator and a novel shock detector for isogeometric analysis.
	Stabilization in time and space directions are performed similarly, which allows us to use high-order discretizations in time without any CFL-like condition.
	The method is proven to yield solutions that satisfy the discrete maximum principle (DMP) unconditionally for arbitrary order. In addition, the stabilization is linearity preserving in a space--time sense. Moreover, the scheme is proven to be Lipschitz continuous ensuring that the nonlinear problem is well-posed.
	Solving large problems using a space--time discretization can become highly costly. Therefore, we also propose a partitioned space--time scheme that allows us to select the length of every time slab, and solve sequentially for every subdomain. As a result, the computational cost is reduced while the stability and convergence properties of the scheme remain unaltered. In addition, we propose a twice differentiable version of the stabilization scheme, which enjoys the same stability properties while the nonlinear convergence is significantly improved.
	Finally, the proposed schemes are assessed with numerical experiments. In particular, we considered steady and transient pure convection and convection-diffusion problems in one and two dimensions.
\end{abstract}

\maketitle

\noindent{\bf Keywords:} 
Isogeometric analysis, discrete maximum principle, monotonicity, high-order, space--time

\tableofcontents
%\begin{AMS}
% AMS codes %65N55, 65F08, 65N30, 65Y05, 65Y20
%\end{AMS}

\pagestyle{myheadings}
\thispagestyle{plain}

\begin{acronym}
  \acrodef{fe}[FE]{finite element}
  \acrodefplural{fe}[FEs]{finite elements}
	\acro{DOF}{degrees of freedom}
	\acro{SSP}{strong stability preserving}
	\acro{DMP}{discrete maximum principle}
	\acro{MP}{maximum principle}
	\acro{LED}{local extremum diminishing}
	\acro{DLED}{discrete local extremum diminishing}
\end{acronym}

\section{Introduction}
Many different applications in science and industry require solving problems satisfying some sort of positivity or \ac{MP} property. These include scalar transport problems, compressible flows, or fluid-based MHD simulations, among others. These problems are of particular interest in a variety of industries and scientific research areas, such as the chemical industry, aviation, aerospace, or nuclear fusion research, just to cite few examples.

Some of these problems exhibit a multiscale nature in time. In those cases, explicit methods are not suitable, since the smallest time scales pose very stringent stability conditions to the time step length, i.e., fully resolved time simulations are required. Thus, implicit methods are favored in applications where the smallest time scales are not of scientific or engineering interest.

As a result, schemes that preserve monotonicity (or at least positivity) for implicit time integration are of special interest. The standard technique to attain such schemes is adding nonlinear artificial diffusion (usually called shock capturing). The common ingredients of a shock capturing or nonlinear stabilization method are the following. The first ingredient is the \emph{artificial diffusion}, which needs to be sufficient to eliminate non-physical oscillations. The schemes in \cite{burman_nonlinear_2002,burman_nonlinear_2007,badia_monotonicity-preserving_2014,badia_discrete_2015} use an element-based artificial diffusion with a standard PDE-based diffusion operator. The drawback of this choice is the fact that the \ac{DMP} only holds under unpractical mesh restrictions. This problem has been solved by Guermond and Nazarov in \cite{guermond_entropy_2011,guermond_maximum-principle_2014} by replacing the PDE-based diffusion operator by an edge or graph-theoretic diffusion operator; see \cite{badia_monotonicity-preserving_2017,badia_differentiable_2017,kuzmin_gradient-based_2016,Lohmann2017,Mabuza2018} for schemes that preserve the \ac{DMP} on arbitrary meshes using a graph-Laplacian. The second ingredient is a \emph{shock detector}, which is the term responsible of deactivating the artificial diffusion in smooth regions. A good shock detector is of vital importance for minimizing the numerical diffusion while satisfying a \ac{DMP}. One example of shock detector is the one developed in 1D by Burman in \cite{burman_nonlinear_2007} and later extended to multiple dimensions by Badia and Hierro \cite{badia_monotonicity-preserving_2014}.  The last ingredient consists on perturbing the mass matrix. One option is a full lumping of the mass matrix, but it can lead to unacceptable phase errors. Instead, a nonlinear lumping is used, e.g., in \cite{badia_differentiable_2017,badia_monotonicity-preserving_2017}, using the same shock detector to lump the mass matrix. Other alternatives can be found in \cite{kuzmin_gradient-based_2016,guermond_correction_2013}. It is worth mentioning that all previous stabilization methods yield a very stiff nonlinear system of equations. In fact, some of the methods proposed in the literature are not even Lipschitz continuous and thus ill-posed (see \cite{barrenechea_edge-based_2016-1}). In practice, the nonlinear convergence of these methods is unacceptably slow, making hard its practical use. To solve this problem, Badia \emph{et al.} \cite{badia_differentiable_2017,badia_monotonicity-preserving_2017} have designed differentiable nonlinear stabilization terms, noticeably improving the nonlinear convergence.

The methods commented above have an algebraic nature and provide some type of \ac{DMP} for the nodal values. The monotonicity of the nodal values only translates into monotonic solutions if the FE space satisfies the convex hull property, which is only true in the first order case. As a result, using the ideas above it does not seem possible to design monotonic second or higher order methods. Recently, Kuzmin and coworkers \cite{Anderson2017,Lohmann2017}, have proposed instead the usage of Bernstein--B\`ezier \acp{fe}, since they satisfy the convex hull for high-order. However, the temporal dimension is discretized using Backward-Euler or \ac{SSP} Runge-Kutta methods (see \cite{ketcheson_optimal_2009}). In the first case, the problem is first order in time, whereas in the second case, a CFL-like condition arises \cite{kuzmin_flux_2002}, since high-order \ac{SSP} methods pose a restriction on the time step size similar to the ones in explicit methods \cite{ketcheson_optimal_2009}.

The main contribution of the present work is the development of a high-order (both in space and time) and \ac{DMP}-preserving discretization for the convection--diffusion--reaction and pure transport problems. This is achieved by combining the nonlinear stabilization techniques in \cite{badia_differentiable_2017,badia_monotonicity-preserving_2017} with a \emph{new shock detector} for arbitrary order space--time isogeometric analysis. 
Another novelty of this study is the stabilization in the time direction, which is performed in a similar manner as in space. This results in an unconditionally stable high-order method in time (and space). However, the space--time method requires to solve the whole space--time problem at once, which increases the computational cost. Hence, we also propose a partitioned approach in the temporal direction, where one can determine the width of the time slab to be computed every time. This strategy allows us to maintain a \emph{reasonable computational cost while having a high-order scheme in space and time, as well as satisfying the \ac{DMP} without any CFL-like condition}. Finally, we also propose a \emph{differentiable version} of the above scheme. This allows us to use Newton's method, which improves nonlinear convergence significantly. 
The method proposed in the present work has been implemented and tested making use of the \texttt{FEMPAR} library \cite{badia_fempar:_2017}.

This paper is structured as follows. First, we introduce the problem, its discretization, and monotonicity properties for scalar problems in Sect. \ref{sec.preliminaries}. Then, the stabilization techniques are introduced in Sect. \ref{sec:isotropic}. Sect. \ref{sec:time-integration} is devoted to the partitioned time integration scheme. Afterwards, we introduce a regularized version of the stabilization term in Sect. \ref{sec:differentiable-stabilization}. 
Finally, we show numerical experiments in Sect. \ref{sec.experiments} and draw some concluding remarks in Sect. \ref{sec.conclusions}.

\section{Preliminaries}\label{sec.preliminaries}
\subsection{Convection-Diffusion problem}
We consider a transient convection-diffusion problem with Dirichlet boundary conditions. Let $\domain \times (0,T) \doteq \prod_{\alpha=1}^{d+1} (0,\side_\alpha)$ be a $(d+1)$-cube, where $d$ is the number of spatial dimensions. Then, the problem reads:
\begin{equation}\label{continuous-problem}
\arraycolsep=1.4pt\def\arraystretch{1.1}
\left\{\begin{array}{rcll}
\partial_t u + \gradient\cdot(\conv u) - \gradient\cdot(\diff \gradient u) & = & \force & \text{in } \domain\times(0,T], \\
u(x,t) & = & \uboundary(x,t) & \text{on } \partial\domain\times(0,T] ,\\
u(x,0) & = & u_0(x) &  x\in\domain,
\end{array}
\right.
\end{equation}
where $\conv$ is a divergence-free convection velocity, $\correction{\diff\geq 0}$ is a scalar constant diffusion, and $\force(\x,t)$ is the body force. In the case of pure convection ($\diff=0$), boundary conditions are only imposed at the inflow $\inflowboundary\doteq\{\x\in\partial\domain : \correction{\conv\cdot\normal_{\partial\domain}<0}\}$, where $\normal_{\partial\domain}$ is a unit vector outward-pointing normal to the boundary. We also define the outflow boundary as $\outflowboundary\doteq\partial\domain\backslash\inflowboundary$. Moreover, we will also consider the steady problem, which is obtained by dropping the time derivative term and the initial \correction{condition}. It is important to mention that a reaction term can be included without harming any of the properties satisfied by the schemes introduced below. However, a convection--diffusion--reaction problem only satisfies a \ac{MP} if the minimum is negative and the maximum positive (analogously for its proposed discretizations). In other words, it only satisfies a weak \ac{MP}, see \cite{barrenechea_edge-based_2016-1}. In order to simplify the discussion below, we will limit the present work to pure convection and convection--diffusion problems.

\correction{In order to avoid technicalities and facilitate the exposition of the stabilization method, we restrict this work to cubic domains. However, it is possible to work with complex geometries using standard procedures from isogeometric analysis \cite{Cottrell2009}. E.g., a complex geometry would be divided in several parts, which would be mapped to multiple $d$-cubic patches. The stabilization method presented in this work is independent from this procedure.}

\subsection{Discretization}\label{sec.discretization}
In this work, we consider a standard B-spline discretization with interpolative boundaries (see \cite{Cottrell2009}). A spline of order $p$ in the variable $x$ is a piecewise polynomial function in $x$ of degree $p$. The values of $x$ in which different polynomials meet are called \emph{knots}. Knots might be placed at the same location, i.e. can be repeated. When the knots are not repeated, the first $p-1$ derivatives of the spline are continuous. When a knot is repeated $r$ times, only the first $p-r$ derivatives are continuous across that knot. Knots are sorted in increasing order and collected in the so called \emph{knot vector} $\{\knots_1, \knots_2 , \ldots\}$. Given a knot vector, B-splines of order $p$ are basis functions for spline functions of the same order. B-splines are constructed in a recursive way using the Cox--de Boor formula:
\begin{equation}
  B_{i}^0(x) \doteq \left\{\begin{array}{ll}
1 & \text{ if }\ \knots_i\leq\knots<\knots_{i+1} \\
0 & \text{ otherwise}
\end{array},\right. \qquad
B_{i}^k(x) \doteq \dfrac{x-\knots_i}{\knots_{i+k}-\knots_i} B_{i}^{k-1}(x) + 
\dfrac{\knots_{i+k+1}-x}{\knots_{i+k+1}-\knots_{i+1}} B_{i+1}^{k-1}(x), 
\end{equation}
for $k=1,...,p$. By construction, $B_i^p(x)$ \correction{has compact support}, is non-negative, and non-zero in $[\knots_i,\knots_{i+p+1}]$. \correction{Notice that its support increases with the degree of the polynomial}.

Let us consider the domain $[0,L]$ and the uniform partition into $m$ sub-intervals of size $h = L / m$. 
The \emph{open} knot vector $\{ \knots_1, \ldots, \knots_{m+2p+1}\}$ is defined as follows. The first $p+1$ knots are located at zero, i.e., $\xi_{1} = \ldots = \xi_{p+1} = 0$. The last $p+1$ knots are located at $L$, i.e., $\xi_{m+p+1} = \ldots = \xi_{m+2p+1}=L$. The interior points are equi-distributed, with $\xi_{i} = (i-p-1) h$, for $i=p+1,\ldots,m+p+1$.  
It leads to a basis $B_i^p(x)$ (for $i=1, \ldots, m+p$) for a space of splines in $[0,L]$ and a partition of unity, i.e., $\sum_{i=1}^m B_i^p(x) = 1$ for $x \in [0,L]$. Any spline $v(x)$ of order $p$ in $[0,L]$ can uniquely be defined by the \emph{control points} $(v_1, \ldots, v_{m+p}) \in \mathbb{R}^{m+d}$ as the linear combination of B-splines $v(x) = \sum_{i=1}^{m+p} B_i^p(x) v_i$. In one dimension, the basis functions obtained from an open knot vector are interpolatory at the extremes, i.e., $v(0) = v_1$ and $v(L) = v_{m+p+1}$ (see Fig. \ref{fig:bspline}). For a first order polynomial $v$ in $[0,L]$, it holds $v(x) = \sum_{i=1}^ {m} B_i^p(x) v(x_i)$, where $x_i \doteq (\knots_{i+1} + \ldots + \knots_{i+p}) / p$ are called the \emph{Greville abscissae} \cite{DeBoor1986,Marsden1970}.

Let us consider the number of partitions per dimension with $m_\alpha$, for $\alpha = 1, \ldots, d+1$. We represent with $\nodes$ the set of multi-indices $\i \doteq (i_1,...,i_{d+1}) \in \mathbb{Z}^{d+1}$ with  $i_\alpha \in \{1, \ldots, m_\alpha + p\}$. Every $\i \in \nodes$ can be expressed as $(\ix,\itemp)$, where $\ix$ is the spatial index and $\itemp$ is the temporal index. The $(d+1)$-dimensional B-spline is defined as the tensor product of $d+1$ unidimensional B-splines ${B}_\i^p(\boldsymbol{x})\doteq B_{i_1}^p(x_1)\times\cdots\times B_{i_{d+1}}^p(x_{d+1})$. Notice that a Greville abscissa in the case of a multidimensional spline reads $\x_\i = (x_{\i_1},...,x_{\i_{d+1}})$. 

We define the space of splines $\fespace \doteq \spn\{{B}_{\i}^p(\x) \ :\ \i\in\nodes\}$.
We use the notation $\shapef[\i]\equiv B_\i^p$. The order is omitted since it is assumed to be fixed. Thus, every spline $\test \in \fespace$ can be written as $\test=\sum_{\i\in\nodes}\shapef[\i]v_\i$. 
Furthermore, we define the following sets of indices, which are useful for the definition of the forthcoming schemes. The set of neighbors of $\i$ is defined as $\stneighborhood[\i] \doteq \{\j\in\nodes : |\i-\j|_\infty \leq 1\}$. We define as $\support{\i}\doteq \{\j\in\nodes : |\i-\j|_\infty \leq p\}$ the set of indices whose associated shape functions intersect with the support of $\shapef[\i]$.

\begin{figure}[h]
	\centering
	\includegraphics[width=0.85\textwidth]{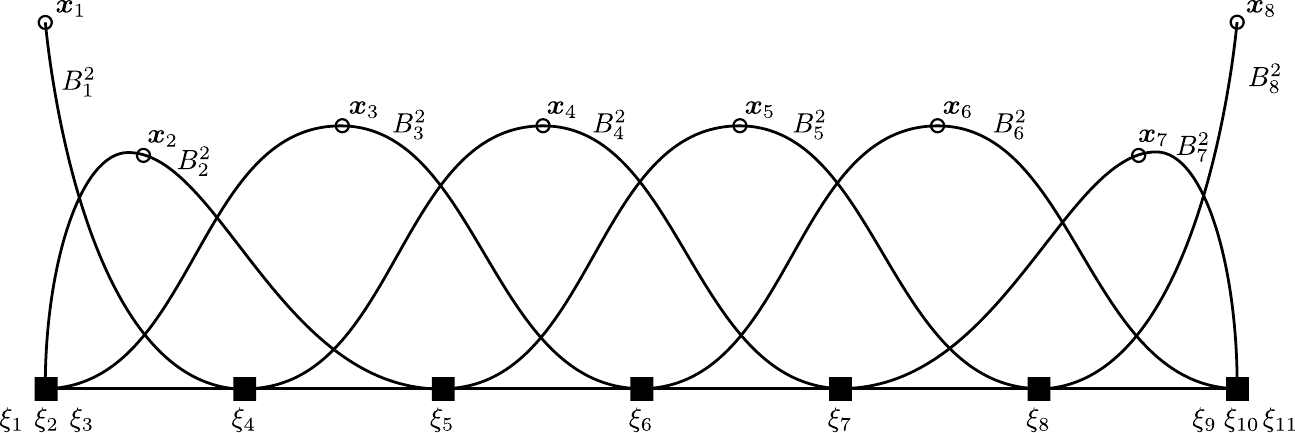}
	\caption{Representation of the basis functions of $\fespace^2$ in one dimension, with its associated Greville abscissae.}
	\label{fig:bspline}
\end{figure}

We use standard notation for Sobolev spaces. The $L^2(\omega)$ scalar product is denoted by $(\cdot,\cdot)_\omega$ for $\omega\subset\domain$. However, we omit the subscript for $\omega\equiv\domain$. The $L^2(\domain)$ norm is denoted by $\Vert\cdot\Vert$. 

\subsection{Discrete problem}
The weak form of \eqref{continuous-problem} using the Galerkin method reads: find $\unk\in\fespace$ such that $\unk(\x,t)=\uboundary_h(\x,t)$ on $\partial\domain \times (0,T]$, $\unk(\x,0) = u_{0h}(\x)$ on $\domain \times \{0\}$, and 
\begin{equation} \label{eq.discrete-problem}
(\partial_t\unk,\test) + (\conv\cdot\gradient\unk,\test) + \diff(\gradient\unk,\gradient\test) = (\force,\test),\qquad \forall \test\in\fespace,
\end{equation}
where $\uboundary_h(t)$ and $u_{0h}$ are projections of $\uboundary(t)$ and $u_0$ to $\fespace$, respectively, such that the local \ac{DMP} is satisfied (see Def. \ref{def.local-dmp}).
Furthermore, we can rewrite the previous discrete problem in matrix form as $\K_\bij u_\j = \F_\i$, where $\K_\bij \doteq (\partial_t\shapef[\j],\shapef[\i]) + (\conv\cdot\gradient\shapef[\j],\shapef[\i]) + \diff(\gradient\shapef[\j],\gradient\shapef[\i])$, and $\F_\i\doteq(\force,\shapef[\i])$ for $\i,\j\in\nodes$. Notice that we have not applied the boundary conditions yet. To apply boundary conditions the space of test functions is restricted to $\test\in\testspace$, and the force vector is redefined as $\F_\i\doteq(\force,\shapef[\i])-(\partial_t\uboundary_h,\shapef[\i]) - {(\conv\cdot\gradient\uboundary_h,\shapef[\i])} - \diff(\gradient\uboundary_h,\gradient\shapef[\i])$.

\subsection{Monotonicity properties}\label{sec.monotonicity-prop}

In this section we define all the properties that we demand our scheme to fulfill. In this case, since we are using a space--time discretization, it becomes more useful to define these properties in a space--time sense. This means that the variation of $\unk$ in the temporal direction will also be taken into account to define an extremum. Hence, we define the concept of a local discrete extremum as follows.
\begin{definition}[Local Discrete Extremum]\label{def.extremum}
%	 Given a time step, $t=t_k$, 
	 The function $\unk\in\fespace$ has a local discrete minimum (resp. maximum) on $\i\in\nodes$, if $u_\i\leq u_\j$ (resp. $u_\i\geq u_\j$) $\forall \j \in\stneighborhood[\i]$.
\end{definition}

\correction{For problems that satisfy a maximum principle, e.g., problem \eqref{eq.discrete-problem} with $g=0$}, it is also important to define the concepts of local and global space--time \ac{DMP}. The latter is a slightly weaker property than the former, but it is more useful in the present study. A local \ac{DMP} is a stronger property because it implies that no oscillations can appear, while the global \ac{DMP} only implies that the global extrema are located at the boundary conditions. 
\begin{definition}[Local space--time \ac{DMP}]\label{def.local-dmp}
	A solution $\unk\in\fespace$ satisfies the local discrete maximum principle if for every $\i\in\nodes$
	\begin{equation}
	\min_{\j\in\stneighborhood[\i]\backslash\{\i\}} u_\j \leq u_\i \leq \max_{\j\in\stneighborhood[\i]\backslash\{\i\}} u_\j.
	\end{equation}
\end{definition}

\begin{definition}[Global space--time \ac{DMP}]\label{def.global-dmp}
	A solution $\unk\in\fespace$ satisfies the global discrete maximum principle if the global extrema are located at boundary conditions, i.e., for every $\i\in\nodes$
	\begin{equation}
	\correction{\min \left(\min_{\x \in \partial \Omega, \  t\in [0,T)} \uboundary_h(\x,t),\ \min_{\x \in \Omega} u_{h0}(\x)  \right) \leq u_\i \leq \max \left( \max_{\x \in \partial \Omega, \ t\in [0,T)} \uboundary_h(\x,t),\ \max_{\x \in \Omega} u_{h0}(\x)\right).}
	\end{equation}

\end{definition}

Finally, let us recall the definition of linearity-preservation, which is a desired property to achieve high-order convergence in smooth regions (see \cite{kuzmin_constrained_2009}). 

\begin{definition}[Linearity-preservation]\label{def.lp}
	A stabilization term, $\B_\bij(\unk)$, is said to be linearity-preserving if, for a solution that is linear in \correction{all directions in the neighborhood of $\x_\i$}, then the stabilization term becomes null, i.e., \correction{$\B_\bij(\unk)=0$ if $\unk(\x)\in\mathcal{P}_1(\domain_\i)$ where $\domain_\i$ is the convex hull defined by the set of neighboring Greville abscissae $\{\x_\j\}_{\j\in \stneighborhood[\boldsymbol{k}],\ \boldsymbol{k}\in\support{\i}}$.}
\end{definition}

\section{Lipschitz-continuous nonlinear stabilization}\label{sec:isotropic}
In this section we define a nonlinear stabilization operator, $B_h(w_h;\unk,\test)$, to be added to the discrete problem \eqref{eq.discrete-problem}, such that it satisfies at least the global \ac{DMP} in Def. \ref{def.global-dmp}. Let us define $\B_\bij(\unk)\doteq B_h(\unk;\shapef[\j],\shapef[\i])$. We also enforce that, for any $\unk\in\fespace$, $\B_\bij(\unk)$ 
\begin{enumerate}
	\item has compact support: $\B_\bij(\unk)=0$ if $\j\not\in\support{\i}$,
	\item is symmetric: $\B_\bij(\unk)=\B_\bji(\unk)$,
	\item is conservative: $\sum_{\j\in \support{\i}\backslash\{\i\}} \B_\bij(\unk)=-\B_{\i\i}(\unk)$.
\end{enumerate}
In order to achieve these requirements, we recall the stabilization term in \cite{badia_monotonicity-preserving_2017}, which is defined as 

\begin{equation}\label{eq.stab-term}
B_h(w_h;\unk,\test)\doteq \sum_{\i\in\nodes}\sum_{\j\in\support{\i}} \artdif_\bij(w_h) v_\i u_\j \graphl(\i,\j),
\end{equation}
for any $w_h, \ \unk, \ \test\in\fespace$. Here, $\graphl$ is the graph-Laplacian operator defined as $\graphl(\i,\j)=2\delta_\bij-1$ (see \cite{guermond_second-order_2014,badia_monotonicity-preserving_2017}), and $\artdif_\bij(w_h)$ is the artificial diffusion defined as
\begin{align}\label{eq.artificial-diffusion}
\artdif_\bij(w_h)&\doteq\max\{\detector[\i](w_h)\K_\bij,0,\detector[\j](w_h)\K_\bji\} \quad \text{for}\quad \j\in\support{\i}\backslash\{\i\}, \\
\artdif_{\i\i}(w_h)&\doteq\displaystyle\sum_{\j\in\support{\i}\backslash\{\i\}} \artdif_\bij(w_h).
\end{align}
We denote by $\detector(w_h)$ the shock detector used for computing the artificial diffusion parameter. The idea behind the definition of this detector is to ensure that the global \ac{DMP} defined in Def. \ref{def.global-dmp} is satisfied using a minimal amount of artificial diffusion, i.e., the lower admissible value of $\artdif_\bij$. A shock detector must be a positive real number, which takes value 1 when $\unk(\x_\i)$ is an inadmissible value of $\unk$ (i.e., local discrete extremum) and smaller than 1 otherwise; to have linearity preservation (see Def. \ref{def.lp}), it must be equal to 0 for $\unk\in\mathcal{P}_1(\Omega\times (0,T])$. In this section, we propose an isotropic approach for $\detector[\i](w_h)$, which consists in using the shock detector in \cite{badia_monotonicity-preserving_2017} in all directions (including time).

\begin{figure}[h]
	\centering
	\includegraphics[width=0.25\textwidth]{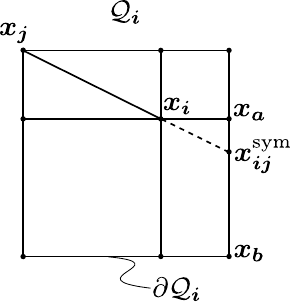}
	\caption{Representation of the polytope $\stdomain_i$ in two dimensions, the symmetric node $\x_\bij^{\rm sym}$ of $\x_\j$ with respect to $\x_\i$, $\x_a$ and $\x_b$.}
	\label{fig:usym}
\end{figure}

In order to introduce the shock detector, let us recall some useful notation from \cite{badia_monotonicity-preserving_2017}. Let $\rr_{\bij} = \x_\j - \x_\i$ be the vector pointing from Greville abscissae $\x_\i$ to $\x_\j$ with $\i,\j\in\nodes$ and $\hat{\rr}_{\bij} \doteq \frac{\rr_{\bij}}{|\rr_{\bij}|}$. Let us take the set of Greville abscissae $\x_\j$ for $\j\in\stneighborhood[i]\backslash\{\i\}$ as vertices of a polytope in $d+1$ dimensions. In particular, let us name this polytope $\stdomain_\i$. Let $\x_\bij^\sym$ be the point at the intersection between $\partial\stdomain_\i$ and the line that passes through $\x_\i$ and $\x_\j$ that is not $\x_\j$ (see Fig. \ref{fig:usym}). The set of all $\x_\bij^\sym$ for all $\j\in\stneighborhood[\i]\backslash\{\i\}$ is represented with $\stsymneigh[\i]$. We define $\rr_\bij^\sym \doteq \x_\bij^\sym - \x_\i$. Given $\x^\sym_\bij$ in two dimensions, let as call $\boldsymbol{a}$ and $\boldsymbol{b}$ the indices of the vertices such that they define the edge in $\partial\stdomain_\i$ that contains $\x_\bij^\sym$. We define $u_\j^\sym$ as the linear interpolation of $u_{\boldsymbol{a}}$ and $u_{\boldsymbol{b}}$ at $\x_\bij$, i.e. $u_\bij^\sym\doteq u_{\boldsymbol{a}}\frac{\x_{\boldsymbol{b}}-\x_\bij^\sym}{\x_{\boldsymbol{b}}-\x_{\boldsymbol{a}}} + u_{\boldsymbol{b}}\frac{\x_\bij^\sym-\x_{\boldsymbol{a}}}{\x_{\boldsymbol{b}}-\x_{\boldsymbol{a}}}$. For higher dimensions, $u_\bij^\sym$ is defined analogously. Given the facet of $\partial\stdomain_i$ where $\x_\bij^\sym$ lies, $u_\bij^\sym$ is the linear interpolation at $\x_\bij^\sym$ of the control points whose Greville abscissae are at the same facet.

Notice that it is essential to use Greville abscissae since they satisfy that for a linear function $\unk\in\mathcal{P}_1$, $\unk(\x_\i)=u_\i$. Therefore, one can construct easily linear approximations of the unknown gradients that are exact for $\unk\in\mathcal{P}_1$. Furthermore,
one can define the jump and the mean of a linear approximation of the unknown gradient at Greville abscissa $\x_\i$ in direction $\rr_\bij$ as
\begin{equation}\label{jump}
\jump{\gradient \unk}_\bij \doteq \frac{u_\j - u_\i}{|\rr_\bij|} + \frac{u_\j^\sym - u_\i}{|\rr_\bij^\sym|}, 
\end{equation}
\begin{equation}\label{mean}
\mean{|\gradient \unk\cdot \hat{\rr}_\bij|}_{\bij} \doteq \frac{1}{2}
\left(\frac{|u_\j-u_\i|}{|\rr_\bij|}+\frac{|u_\j^\sym-u_\i|}{|\rr_\bij^\sym|}\right).
\end{equation}
In the present work we will use the same shock detector developed in \cite{badia_monotonicity-preserving_2017}, which reads
\begin{equation}\label{eq.iso-detector}
\detector[\i](\unk)\doteq\left\{\begin{array}{cc}
\left[\dfrac{\left|\sum_{\j\in\stneighborhood[\i]} \jump{\gradient\unk}_\bij\right|}{\sum_{\j\in\stneighborhood[\i]}2\mean{\left|\gradient\unk\cdot\hat{\rr}_\bij\right|}_\bij}\right]^q & \text{if}\quad \sum_{\j\in\stneighborhood[\i]} \mean{\left|\gradient\cdot\hat{\rr}_\bij\right|}_\bij \neq 0 \\
0 & \text{otherwise}
\end{array}
\right. .
\end{equation}
From Lm. 3.1 in \cite{badia_monotonicity-preserving_2017} we know that \eqref{eq.iso-detector} is valued between 0 and 1, and it is only equal to one if $\unk(\x_i)$ is a local discrete extremum (in a space--time sense as in Def. \ref{def.extremum}).
Since the linear approximations of the unknown gradients are exact for $\unk\in\mathcal{P}_1$, the shock detector vanishes when the solution is linear in all dimensions (including time). This result follows directly from \cite[Th. 4.5]{badia_monotonicity-preserving_2017}.

Supplementing the discrete problem \eqref{eq.discrete-problem} with the above stabilization term, the stabilized problem reads: Find $\unk\in\fespace$ such that $\unk=\uboundary_h$ on $\partial\domain$, $\unk=u_{0h}$ at $t=0$, and
\begin{equation}\label{eq.stabilized-problem}
(\partial_t\unk,\test) + (\conv\cdot\gradient\unk,\test) + \diff(\gradient\unk,\gradient\test)+B_h(\unk;\unk,\test) = (\force,\test),\qquad \forall \test\in\fespace,
\end{equation}
which in turn can be expressed in matrix form as
\begin{equation}\label{eq.matrix-problem}
\sK_\bij(\unk)u_\j = \F_\i,
\end{equation}
where $\sK_\bij(\unk) \doteq \K_\bij + \B_\bij(\unk)$ for $\i,\j\in\nodes$.

\begin{theorem}[\ac{DMP}]\label{thm.genDMP}
	The solution of the discrete problem \eqref{eq.matrix-problem} using the shock detector \eqref{eq.iso-detector} satisfies the global \ac{DMP} in Def. \ref{def.global-dmp} if $\force = 0$ and, for every control point $\i\in\nodes$ such that $u_\i$ is a local discrete extremum, it holds:
	\begin{equation}\label{dmp-cond}
	\sK_\bij(\unk) \leq 0,  \,\forall \, \j \in \support{\i} \backslash\{\i\},\qquad \sum_{\j\in\support{\i}} \sK_\bij(\unk)  = 0.
	\end{equation}
	Moreover, the resulting scheme is linearity-preserving as defined in Def. \ref{def.lp}, i.e. $\B_\ij(\unk)=0$ for $\unk\in\mathcal{P}_1$.
\end{theorem}
\begin{proof}
	Let us assume that $u_\i$ is a discrete maximum. Then, \eqref{eq.matrix-problem}, for $\force=0$ and before applying boundary conditions, reads 
	\begin{equation}
	\sum_{\j\in\support{\i}} \sK_\bij(\unk)u_\j = \0.
	\end{equation}
	
	Therefore, $u_\i$ can be computed as
	\begin{equation}
	u_\i = \frac{\sum_{\j\in\support{\i}\backslash\{\i\}} \sK_\bij(\unk)u_\j}{\sK_{\i\i}(\unk)} .
	\end{equation}
	Since $u_\i$ is an extremum, which implies $\detector[\i]=1$, the stabilization term ensures \eqref{dmp-cond} by construction.
	Hence, the coefficients 
	that multiply $u_\j$ are in $[0,1]$, and the sum of all 
	these coefficients add up to one. Therefore, $u_\i$ is a convex
	combination of its neighbors (including boundary conditions
	$\uboundary$). Since $u_\i$ is a maximum and a convex combination of its
	neighbors, then $u_{\boldsymbol{k}}=u_\i$ for some $\boldsymbol{k}\in\support{i}$. \correction{In that case, we can write
	\begin{equation}
	u_\i = \frac{\sK_{\i\boldsymbol{k}}(\unk)u_{\boldsymbol{k}}}{\sK_{\i\i}(\unk)} + \frac{\sum_{\j\in\support{\i}\backslash\{\i,\boldsymbol{k}\}} \sK_\bij(\unk)u_\j}{\sK_{\i\i}(\unk)}.
	\end{equation}
	Since $u_{\boldsymbol{k}}=u_\i$,
	\begin{equation}
	\left(1-\frac{\sK_{\i\boldsymbol{k}}(\unk)}{\sK_{\i\i}(\unk)}\right)u_\i =  \frac{\sum_{\j\in\support{\i}\backslash\{\i,\boldsymbol{k}\}} \sK_\bij(\unk)u_\j}{\sK_{\i\i}(\unk)} .
	\end{equation}
	Therefore, it can also be proved that $u_\i$ is a convex combination of all its
	neighbors \emph{but} $u_{\boldsymbol{k}}$,
	\begin{equation}
	u_\i =  \frac{\sum_{\j\in\support{\i}\backslash\{\i,\boldsymbol{k}\}} \sK_\bij(\unk)u_\j}{\left(1-\frac{\sK_{\i\boldsymbol{k}}(\unk)}{\sK_{\i\i}(\unk)}\right)\sK_{\i\i}(\unk)} = \frac{\sum_{\j\in\support{\i}\backslash\{\i,\boldsymbol{k}\}} \sK_\bij(\unk)u_\j}{\sK_{\i\i}(\unk)- \sK_{\i\boldsymbol{k}}(\unk)} .
	\end{equation}
	Proceeding analogously, one can also prove that $u_{\boldsymbol{k}}$ 
	is a convex combination of all its neighbors \emph{but} $u_\i$. Hence, 
	we know that the value of $u_\i=u_{\boldsymbol{k}}$ is bounded by all
	their neighbors. At this point, the same reasoning can be applied to any of
	their neighbors. Thus,} by induction, we know that extrema at any control point 
	are bounded by the boundary conditions. Thus, the global \ac{DMP} is satisfied.
	
	From \cite[Th. 4.5]{badia_monotonicity-preserving_2017}, \correction{$\detector[\j]=0$ for any $\j\in\support{\i}$ if $\unk\in\mathcal{P}_1(\domain_\i)$ where $\domain_\i$ is the convex hull defined by the set of neighboring Greville abscissae $\{\x_\j\}_{\j\in \stneighborhood[\boldsymbol{k}],\ \boldsymbol{k}\in\support{\i}}$. By definition, the stabilization term also vanishes if $\detector[\j]=0$ for $\j\in\support{\i}$ (see \eqref{eq.stab-term} and \eqref{eq.artificial-diffusion})}. Therefore,  the scheme is linearity-preserving as defined in Def. \ref{def.lp}.	
\end{proof}
\begin{theorem}
	The diffusion defined in \eqref{eq.artificial-diffusion} introduces the minimal amount of artificial dissipation such that condition  \eqref{dmp-cond} is satisfied when $q\rightarrow\infty$.
\end{theorem}
\begin{proof}
	The proof follows the same lines as in \cite[Th. 4.4]{badia_monotonicity-preserving_2017}. We do not include it for the sake of conciseness.
\end{proof}

Finally, we prove Lipschitz continuity of the stabilization term. 
%either when using isotropic \eqref{eq.iso-detector} or space--time \eqref{eq.st-detector} detector. 
In this particular case, the proof follows the same reasoning as in \cite{badia_monotonicity-preserving_2017}. \correction{Let us recall the definition of the semi-norm generated by the graph-Laplacian required to show Lipschitz continuity,
\begin{equation}
|w|_\ell \doteq \sqrt{\half\sum_{\i\in\stneighborhood}\sum_{j\in\support{\i}} (w_\i - w_\j)^2 }.
\end{equation}
In addition, we will also need the $\ltwo(\domain)$ norm denoted as $\|\cdot\|$ and the $\linfty(\domain)$ expressed as $\|\cdot\|_{\infty}$.}
We do not include all details for the sake of conciseness and refer the reader to the previously cited work. 

\begin{theorem}\label{thm.lipcont}
	\correction{Let $\fespace^{\rm adm}\subset\fespace$ be the subspace of functions that satisfy the global \ac{DMP} in Def.\ \ref{def.global-dmp}}, then $\B(\cdot)$ with the shock detector \eqref{eq.iso-detector} is Lipschitz continuous in $\fespace^{\rm adm}$ for $\unk\in\fespace$ and bounded $q$, since \correction{
	\begin{equation}
	(\B(u)-\B(v),z) \leq Cq(h^d+\|\conv\|_{\infty}h^{d-1}\delta t+\diff h^{d-2}\delta t)|u-v|_\ell |z|_\ell 
	\end{equation}}
	is satisfied.
\end{theorem}
\begin{proof}
	The proof is analogous to the one in \cite[Th. 6.1]{badia_monotonicity-preserving_2017}. The only difference arises from the bound for
	\begin{equation}
	\K_\bij = (\partial_t\shapef[\j],\shapef[\i]) + (\conv\cdot\gradient\shapef[\j],\shapef[\i]) + \diff(\gradient\shapef[\j],\gradient\shapef[\i]).
	\end{equation}
	In this case, using Cauchy-Schwarz inequality, the inverse inequality $\|\gradient v_h\|\leq C h^{-1}\|v_h\|$ for $v_h \in\fespace$, and $\|\shapef[i]\|\leq Ch^{d/2}$, we get\correction{
	\begin{align}
	\K_\bij &\leq \|\partial_t\shapef[\j]\|\|\shapef[\i]\| + \|\conv\|_{\infty}\|\gradient\shapef[\j]\|\|\shapef[\i]\| + \diff \|\gradient\shapef[\j]\|\|\gradient\shapef[\i]\| \\
	&\leq C_1 h^{d}
    + \|\conv\|_\infty C_2 h^{d-1}\delta t + \diff C_3 h^{d-2}\delta t ,
	\end{align}}
	where $d$ is the number of spatial dimensions, $h$ is the distance between knots for the spatial directions and $\delta t$ is the distance between knots for the time direction.
\end{proof}
\correction{
Notice that whereas $C$ is uniform with respect to the mesh size, it can depend on the polynomial order of the discretization.}

\section{Time partitioned scheme}\label{sec:time-integration}
Hitherto, we have only considered the solution of the whole space--time problem at once. In order to substantially reduce the computational cost, we propose the division of the time integration in several time subdomains, considering the proposed space--time formulation at every subdomain. Namely, the problem $\eqref{continuous-problem}$ set in $\domain\times(0,T]$, will be decomposed in $\domain\times(t_l,t_{l+1}]$ for $0\leq l \leq n_t-1$, with $t_0=0$ and $t_{n_t}=T$. We define the length of the subdomain as $\Delta t\doteq t_{l+1}-t_l$, and restrict its possible values to $\Delta t = n\, p\, \delta t$ for some $n\in\mathbb{N}$, where $p$ is the order of the spline space, and $\delta t$ is the distance between knots in the temporal direction. Notice that we are only using discretizations formed from the tensor product of discretizations in 1D. Therefore, with the particular choice of $\Delta t$, $t_l$ will always be the temporal coordinate of a layer of knots. Hence, performing this kind of partitions is straightforward. Other partitions might be considered, however we choose the previous one because it is particularly simple to use it in our implementation.

%\JB{reduce/simplify}
The approximation space of splines for every subdomain is obtained as follows. Given the complete domain $\domain\times(0,T]$, we discretize it as described in Sect. \ref{sec.discretization}, resulting in a spline space $\fespace$. Then in order to reduce the coupling between partitions, we insert $p$ knots at $t_l$. The resulting spaces at each subdomain, say $\fespace^l$, are fully decoupled. However, due to causality in time there exists a sequential coupling between subdomains, i.e. the information travels in the positive direction. In other words, the solution at subdomain $l$ will affect the solution at $l+1$, but not the opposite. Therefore, we impose that the initial conditions at subdomain $l+1$ are equal to the solution at the final time of subdomain $l$, i.e. $u_h^{l+1}(t_l)\doteq u_h^{l}(t_l)$. After imposing this restriction, the complete approximation space, $\tilde{\fespace}$, is $C^0$, and coupled sequentially. Hence, each subdomain can be solved sequentially, and thus the computational cost is significantly reduced. 

The partitioned space--time scheme with nonlinear stabilization reads as follows.
For $l=1,...,n_t$; find $\unk^l\in\fespace^l$ such that $\unk^l=\uboundary_h$ on $\partial\domain$, $\unk^l(\x,t_l)=\unk^{l-1}(\x,t_l)$ with $\unk^0(\x,t_0) = u_{0h}$, and
\begin{equation}\label{time-discrete-problem}
(\partial_t\unk^l,\test) + (\conv\cdot\gradient\unk^l,\test) + \diff(\gradient\unk^l,\gradient\test) + B_h(\unk^l;\unk^l,\test) = (\force,\test),\qquad \forall \test\in\fespace^l,
\end{equation}

Due to the partition $\unk$ will only be piecewise continuous in time. Let us prove now that the scheme still satisfies the global \ac{DMP}.

\begin{lemma}
	The solution of problem \eqref{time-discrete-problem}, using the shock detector defined in \eqref{eq.iso-detector}, satisfies the global \ac{DMP} (Def. \ref{def.global-dmp}) if $g=0$ in $\domain\times(0,T]$ and, for every control point $\i\in\nodes$ such that $u_\i$ is a local discrete extremum, conditions \eqref{dmp-cond} hold.
\end{lemma}
\begin{proof}
	From Th. \ref{thm.genDMP} it is easy to see that conditions \eqref{dmp-cond} hold for the first subdomain. Hence, the solution at the first subdomain at time $t_1$, $\unk^1(\x,t_1)$, is bounded by the initial and boundary conditions. Since $\unk^1(\x,t_1)$ are the initial conditions for the second subdomain, then again from Th. \ref{thm.genDMP} it is known that the solution in the second subdomain is bounded by $\unk^1(\x,t_1)$, and thus by the initial and boundary conditions. Therefore, by induction, we conclude that the global \ac{DMP} is satisfied in the whole domain.
\end{proof}

\section{Differentiable stabilization}\label{sec:differentiable-stabilization}

In this section, we introduce a different version of the previous operators.
As exposed in \cite{badia_differentiable_2017,badia_monotonicity-preserving_2017}, the regularization of all non-differentiable operators in the stabilization term improves the nonlinear convergence, and allows us to use Newton's method. We use the same strategy introduced in \cite{badia_monotonicity-preserving_2017}. Then, the shock detector reads

\begin{equation}\label{eq.smthdetector}
\smthdetector[i](\unk)\doteq
\left[\smthlimit{
	\dfrac{\absn{\sum_{\j\in\neighborhood[\i]} \jump{\gradient\unk}_\bij}+\gamma_h}{\sum_{\j\in\neighborhood[\i]}2\mean{\absd{\gradient\unk\cdot\hat{\rr}_\bij}}_\bij+\gamma_h}}\right]^q,
\end{equation}
where $\gamma_h>0$ is a parameter to prevent division by zero, and the regularized absolute values by 
\begin{equation}
\absn{x} =  \sqrt{x^2 + \varepsilon_h},\qquad
\absd{x} =  \frac{x^2}{\sqrt{x^2 + \varepsilon_h}}.
\end{equation}
Notice that $\absd{x}\leq |x|\leq\absn{x}$. With this regularization, the quotient in the shock detector might become greater than one, thus we need to smoothly limit its value to one. To this end we recall $\smthlimit{x}$, which reads
\begin{equation}
\smthlimit{x} \doteq \left\{\begin{array}{ll}
2x^4-5x^3+3x^2+x & x<1\\
1 & x\geq 1
\end{array}\right. ,
\end{equation}
and clearly is twice differentiable and bounded above by 1.
The differentiable version still satisfies the requirements for a shock detector, i.e., it is a real value in $[0,1]$ and it is equal to 1 if $u_\i$ is a local extrema. This result follows directly from \cite[Lm 7.1]{badia_monotonicity-preserving_2017}.

Furthermore, for the definition of the artificial diffusion we need to regularize the maximum function. We choose again the same strategy as in \cite{badia_monotonicity-preserving_2017}, and define $\smax\{\cdot,\cdot\}$ as
\begin{equation}\label{eq:smax}
\smax\{x,y\} \doteq \frac{ \absn[\sigma_h]{x -y}}{2} +\frac{x +y}{2}.
\end{equation}
Finally, we can define the twice differentiable artificial diffusion parameter as
\begin{equation}\label{smth-artificial-diffusion}
\smthartdif_\bij(w_h)\doteq\left\{
\eqarrayconf
\begin{array}{l}
\smax\left\{\smax\left\{\smthdetector[\i](w_h)\K_\bij,\smthdetector[\j](w_h)\K_\bji\right\},0\right\} \quad \text{for}\quad \j\neq \i \\
\displaystyle\sum_{\j\in\support{\i}\backslash\{\i\}} \smthartdif_\bij(w_h)
\end{array}\right. ,
\end{equation}
and the stabilization operator reads
\begin{equation}
\tilde{B}_h(w_h;\unk,\test)\doteq \sum_{\i\in\nodes}\sum_{\j\in\support{\i}} \smthartdif_\ij(w_h) v_\i u_\j \graphl(\i,\j).
\end{equation}

In order to obtain a differentiable operator, we have added a set of regularizations that rely on different parameters, e.g., $\sigma_h,\ \varepsilon_h,\ \gamma_h$. Giving a proper scaling of these parameters is essential to recover theoretic convergence rates. In particular, we use the following relations
\begin{equation}
\sigma_h=\sigma|\bsbeta|^2  L^{2(d-3)} h^{2(p+1)},\quad \varepsilon_h=\varepsilon L^{-4} h^2,\quad \gamma_h=L^{-1}\gamma,
\end{equation}
where $d$ is the spatial dimension of the problem, $L$ is a characteristic length, and $\sigma,\ \varepsilon,$ and $\gamma$ are of the order of the unknown.

\section{Numerical experiments}\label{sec.experiments}

In this section we present some numerical experiments showing the behavior of the scheme previously introduced. First, a convergence analysis is performed in order to assess the correctness of the proposed scheme and its implementation. Then, we assess the performance of the proposed stabilization method for high-order discretizations, including a brief analysis of the effect of the regularization.

\subsection{1D Transient Diffusion}\label{ssec.test1} % No stab. convergence test time partitions
The purpose of this test is assessing the partitioned time integration scheme in Sect. \ref{sec:time-integration}. To this end, we solve the following problem for $t\in(0,1]$ and $x\in\domain\doteq(0,1)$,
\begin{equation}\label{eq.transient-diffusion}
\left\{\begin{array}{rl}
\partial_t u + \partial_{xx}u = f & \text{in } \domain\times(0,1] \\
u = 0 & \text{at } \partial\domain
\end{array}\right.,
\end{equation}
where $f\doteq 2(6x^2 - 6x + 1)(t(t-1))^2 + 2t(t-1)(2t-1)(x(x-1))^2$. This problem has $u=(x(x-1))^2 \ (t(t-1))^2$ as exact solution. We perform a convergence analysis where the mesh is successively refined in the time direction for first, second, and third order discretizations. In particular, the distance between knots in the temporal direction is refined as ${\delta t=\{0.2, 0.1,0.05,0.025,0.0125\}}$ for the first order discretization. In spatial directions, the distance is small enough $(h = 1/400)$ to prevent that spatial discretization errors affect the analysis. Second and third order discretizations are obtained using the following $k$-refinement (see \cite{Cottrell2009} for more details). We refine the discretization such that the number of control points increase at the same rate as a Lagrangian \ac{fe} discretization does when its order is increased. Fig.~\ref{fig.k-rfinement} shows the result of $k$-refinements to $p=2$ and $p=3$ discretizations, for an interior subset of the discretization. Henceforth, we will use this kind of $k$-refinement in order to increase the discretization order.

\begin{figure}[h]
	\centering
	\includegraphics[width=0.85\textwidth]{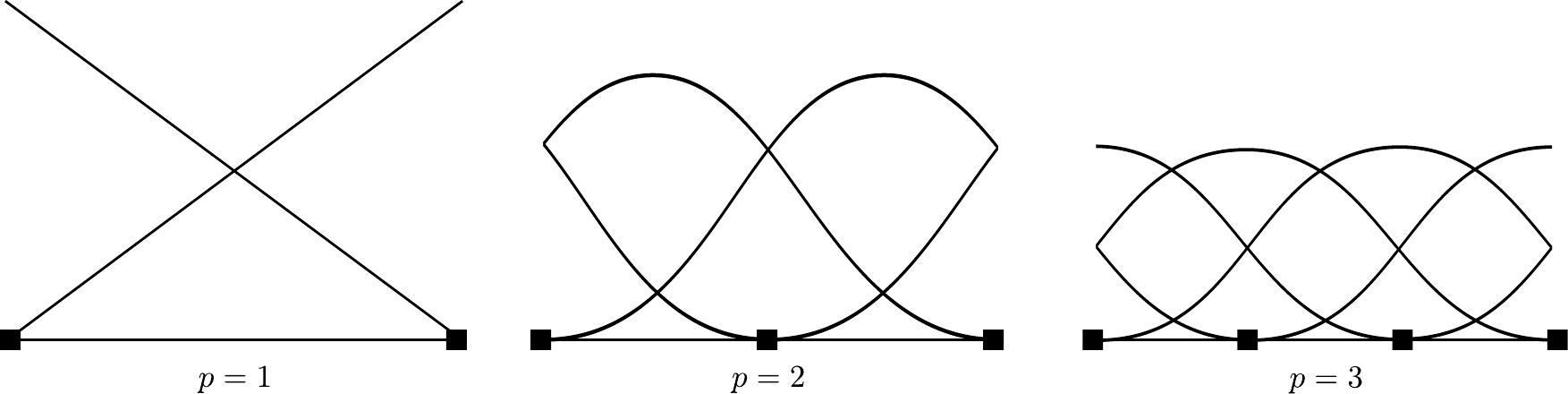}
	\caption{Second and third order discretizations obtained from the $k$-refinement of an initial first order discretization. Notice that shape functions are depicted for interior knots, at boundary knots shape functions become interpolatory, see Fig. \ref{fig:bspline}.}
	\label{fig.k-rfinement}
\end{figure}

\begin{figure}[h]
	\centering
	\begin{subfigure}[b]{0.48\textwidth}
		\includegraphics[width=\textwidth]{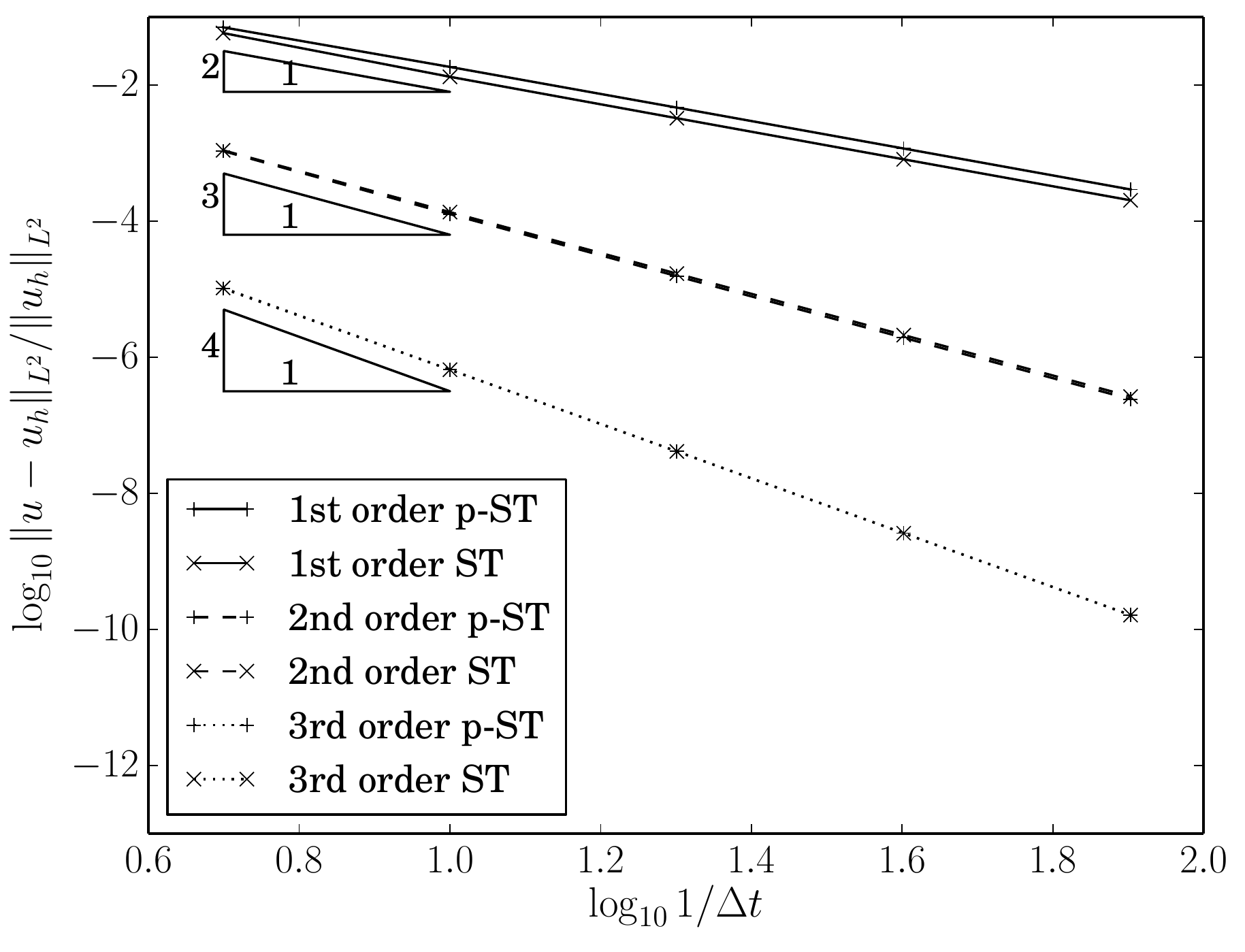}
		\caption{Relative $\ltwo$ norm of the error.}
	\end{subfigure}
	\hfill
	\begin{subfigure}[b]{0.48\textwidth}
		\includegraphics[width=\textwidth]{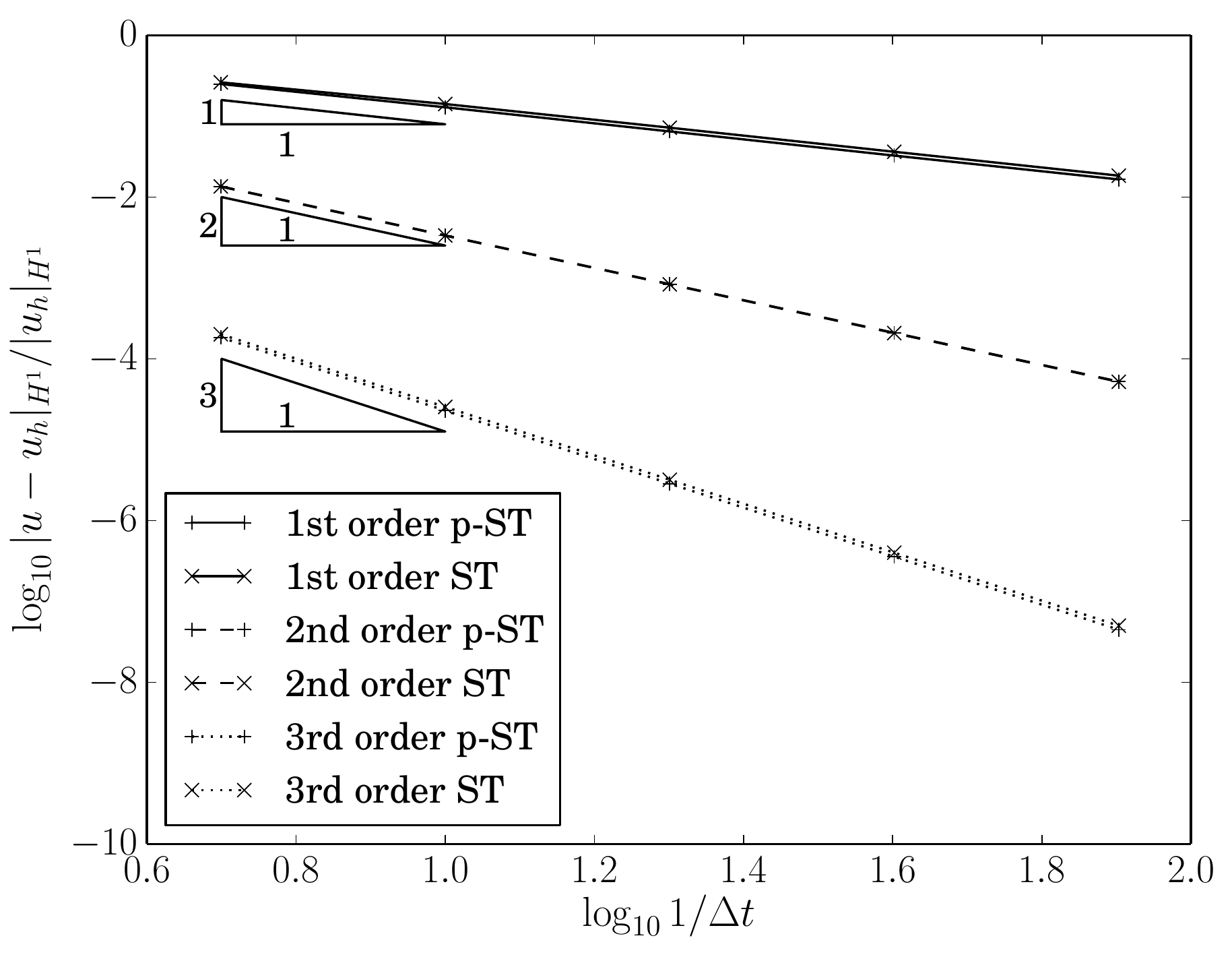}
		\caption{Relative $\hone$ norm of the error.}
	\end{subfigure}
	\caption{Convergence in time results for problem \eqref{eq.transient-diffusion}, using standard and partitioned space--time schemes.}
	\label{fig.convergence-rate}
\end{figure}

We measure the relative $\ltwo$ norm and $\hone$ semi-norm of error \correction{in the whole space--time domain}, and compute the resulting convergence rate. 
Errors in $\ltwo$ norm and $\hone$ semi-norm are depicted in Fig. \ref{fig.convergence-rate} (a) and (b), respectively.
In Table \ref{tab.1d-conv-test} the measured convergence rates are shown for the original non-partitioned scheme and the proposed in Sect. \ref{sec:time-integration}. We observe a slight increase in the error for the partitioned scheme. However, the obtained results show optimal convergence rates for both schemes introduced above.

\begin{table}[h]
\centering
\caption{Measured convergence rates in $\ltwo$ norm and $\hone$ semi-norm, for problem \eqref{eq.transient-diffusion}.}
\label{tab.1d-conv-test}
\begin{tabular}{cccc} \hline
	Order & Method & $\ltwo$ convergence & $\hone$ convergence \\ \hline
	1   & p-ST & -1.98  & -0.98 \\
	1   & ST   & -2.04  & -0.96 \\
	2   & p-ST & -3.03  & -2.00 \\
	2   & ST   & -3.00  & -2.01 \\
	3   & p-ST & -3.99  & -3.00 \\
	3   & ST   & -3.99  & -2.99 \\ \hline
\end{tabular}\\
{\small p-ST: Partitioned space--time, ST: space--time.}
\end{table}

\subsection{Steady convection} % Convergence test monotonic + nonmonotonic

In this experiment we assess the convergence of the stabilized schemes introduced in Sect. \ref{sec:differentiable-stabilization}. We use a steady pure convection problem with a non-monotonic smooth solution. In particular, we solve the following problem for $x\in\domain\doteq[0,1]^2$, 
\begin{equation}\label{eq.nonmonotonic}
\left\{\begin{array}{rl}
\conv\cdot\gradient u= 0 & \text{in } \domain \\
u = u_D & \text{at } \inflowboundary
\end{array}\right.,
\end{equation}
where $u_D = \sin\left(2 \pi \left(x-\frac{y}{\tan \theta}\right)\right)$, $\conv = (\cos \theta,\sin\theta)$, and $\theta = \pi/3$. The analytical solution of the above problem reads 
$u = \sin \left(2\pi \left(x-\frac{y}{\tan \pi/3}\right)\right)$. The convergence analysis is performed for first, second, and third order discretizations\correction{, i.e. $p=\{1,2,3\}$}. We use a standard nonlinear solver (see \cite{badia_differentiable_2017} for details), with a nonlinear tolerance $\frac{u^{k+1}-u^k}{u^k}<10^{-6}$. The following stabilization parameters have been used: $q=10$, $\varepsilon=10^{-5}$, $\sigma=10^{-6}$, $\gamma=10^{-10}$. The selection of these parameters is based on the outcome of previous works \cite{badia_differentiable_2017,badia_monotonicity-preserving_2017} and Sect. \ref{sec.params}.

Convergence plots are shown in Fig. \ref{fig.space-convergence-rate} and the corresponding convergence rates in Table \ref{tab.space-conv-test}. As expected, it is observed that the scheme recovers second order convergence in the $\ltwo$ error norm and first in the $\hone$ error semi-norm. 
It is known that the stabilized scheme should recover second order convergence for $p=1$. However, due to peak clipping errors, higher convergence rates are not expected even if a higher order discretization is used \cite{Kuzmin2017}.
In any case, we do observe that the error diminishes as the discretization order is increased using the $k$-refinement previously defined. 

\begin{figure}[h]
	\centering
	\begin{subfigure}[b]{0.48\textwidth}
		\includegraphics[width=\textwidth]{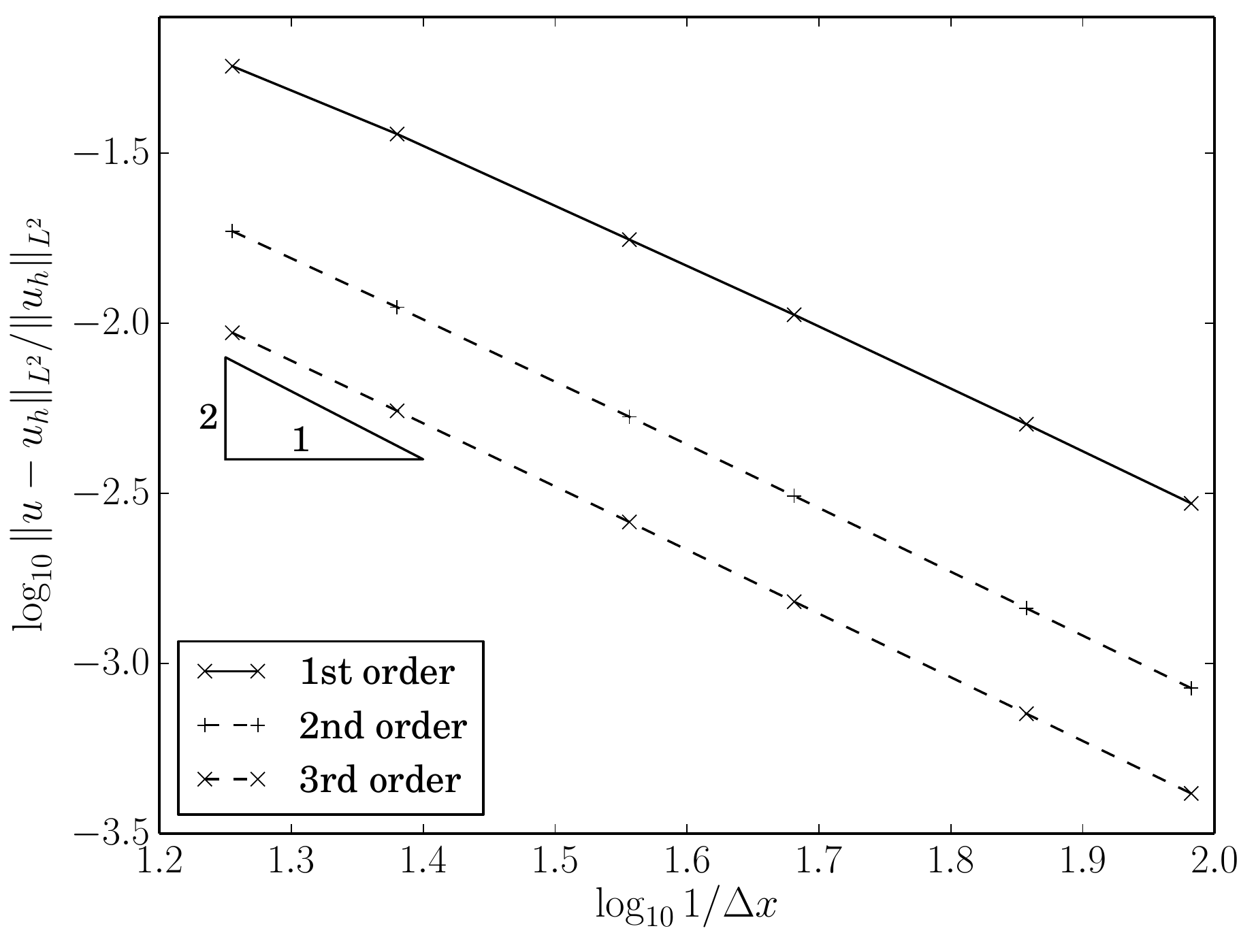}
		\caption{Relative $\ltwo$ norm of the error.}
	\end{subfigure}
	\hfill
	\begin{subfigure}[b]{0.48\textwidth}
		\includegraphics[width=\textwidth]{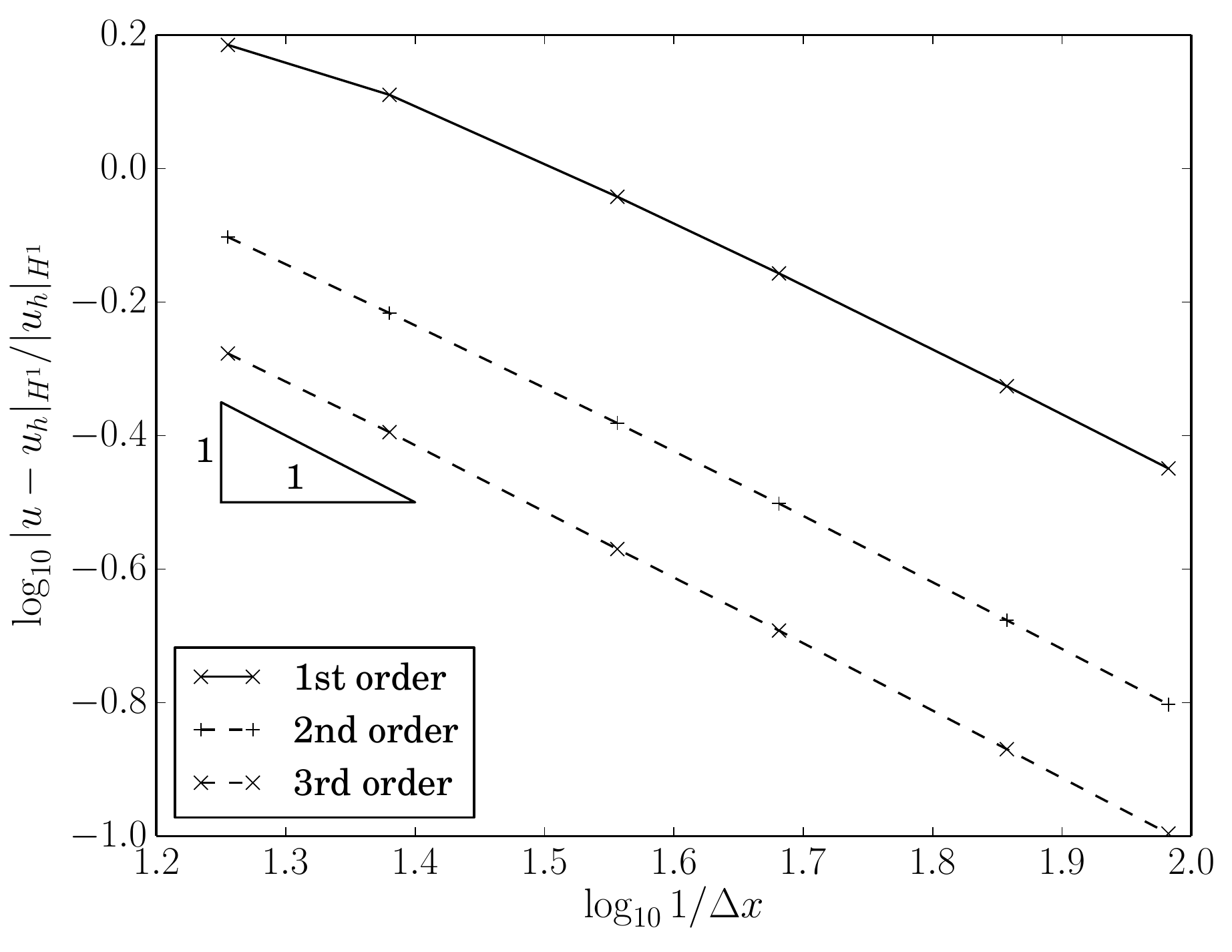}
		\caption{Relative $\hone$ semi-norm of the error.}
	\end{subfigure}
	\caption{Convergence in space results for problem \eqref{eq.nonmonotonic}.}
	\label{fig.space-convergence-rate}
\end{figure}

\begin{table}[h]
	\centering
	\caption{Measured convergence rates in $\ltwo$ and $\hone$ norms, for problem \eqref{eq.nonmonotonic}.}
	\label{tab.space-conv-test}
	\begin{tabular}{cccc} \hline
		Order & $\ltwo$ convergence & $\hone$ convergence \\ \hline
		1    & 1.77  & 0.88  \\
		2    & 1.85  & 0.96  \\
		3    & 1.86  & 0.99  \\ \hline
	\end{tabular}
\end{table}

\subsection{Nonlinear convergence}\label{sec.params}
In the current test, we aim to briefly analyze the effect of the stabilization parameters on the nonlinear convergence of the method. To this end, we solve the following 1D pure convection problem with discontinuous initial conditions.
\begin{equation}\label{eq.sharplayer}
\left\{\begin{array}{rl}
\partial_t u + \conv\cdot\gradient u= 0 & \text{in } \domain\times[0,T) \\
u = u_0 & \text{at } t=0 \\
u = u_D & \text{at } \partial\domain
\end{array}\right.,
\end{equation}
where $\conv\doteq 1$, $\domain\doteq(0,1]$, $T=0.5$, and $u_0\doteq 1 - H_0(x-0.25)$, where $H_0$ is the well-known zero-centered Heaviside function. First and second discretization orders are used in a coarse mesh of $25\times 25$ control points. To obtain the second order mesh, we perform the $k$-refinement as in the previous experiment.

We refer the reader to \cite{badia_differentiable_2017,badia_monotonicity-preserving_2017} for a deeper analysis on the effect of each regularization parameter. Therein, the same family of shock detectors is used in the context of first order cG and dG Lagrangian FEs. In the present study, we analyze the effect of the regularization globally using a fixed relation between the different parameters. In particular, we use the following parameters: $\gamma=10^{-10}$, $\sigma=\zeta$, $\varepsilon=\zeta^2$, where $\zeta=\{10^{-1}, 10^{-2}, 10^{-3}, 10^{-4}\}$. Furthermore, the effect is also compared as $q$ is incremented, particularly for $q=\{1,2,5,10\} $. In addition, the non-regularized version is also used to show the improvement in the nonlinear convergence.
The relaxed Picard and hybrid nonlinear solvers presented in \cite{badia_differentiable_2017} are used, \correction{and the nonlinear tolerance is set to $10^{-5}$.}

Fig. \ref{fig.nl-conv-p1} and \ref{fig.nl-conv-p2} show the results for first and second order discretizations, respectively. In general terms, as $q$ is increased or $\zeta$ is decreased, sharper solutions are observed. However, nonlinear iterations increase. As expected, the hybrid method outperforms the relaxed Picard method. Even though it requires more nonlinear iterations, the non-regularized detector might be a simpler (parameter-free) alternative to the regularized one. Finally, it is worth mentioning that a slight increase in the required number of iterations is observed as the discretization order is increased. However, the obtained results are more accurate, i.e., the discontinuity becomes sharper.

\begin{figure}[h]
	\centering
	\includegraphics[width=\textwidth]{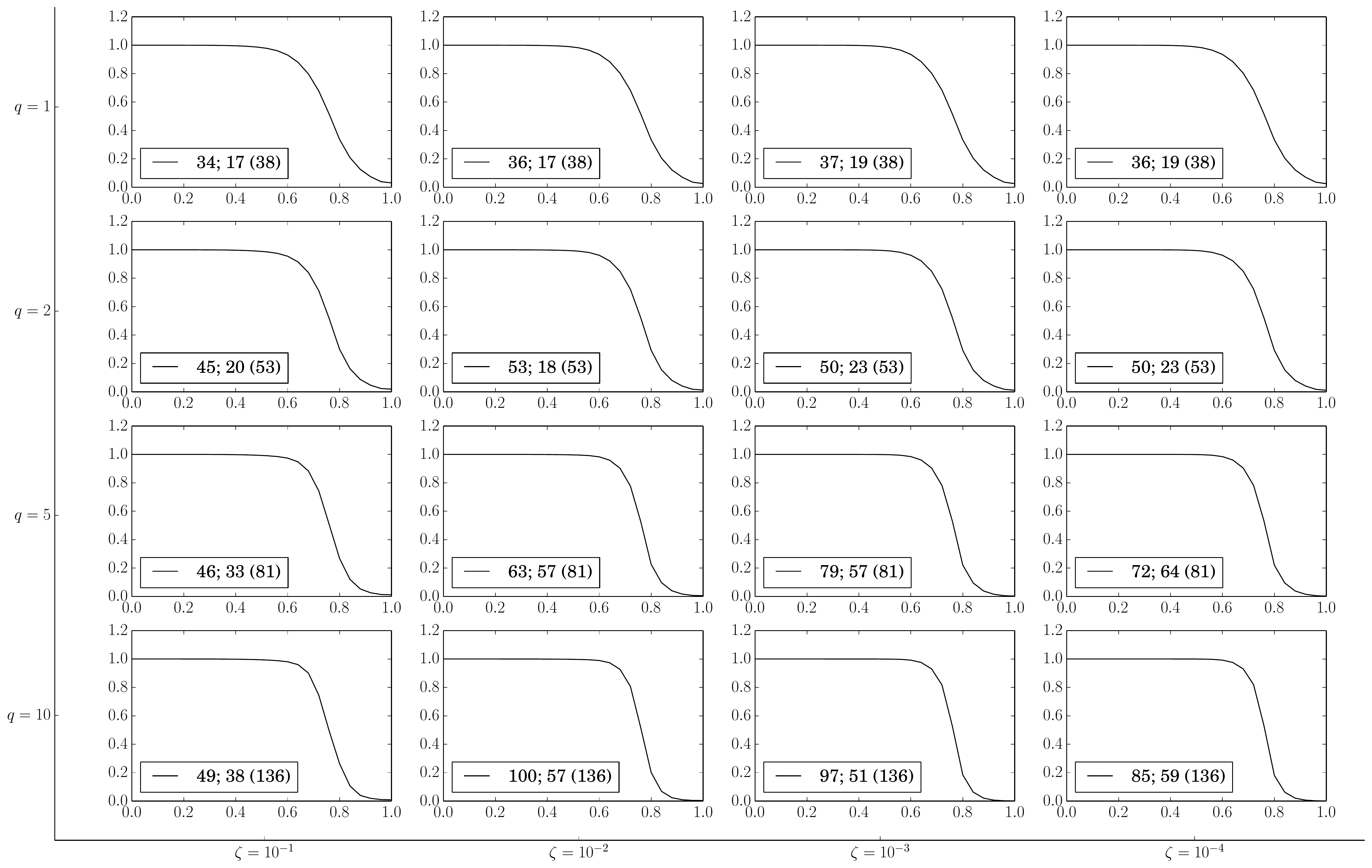}
	\caption{Effect of the regularization parameters for first order discretizations. The numbers in legends are the number of nonlinear iterations performed. First number is for relaxed Picard and the next for hybrid scheme, both for the regularized stabilization. The number in brackets is the number of iterations required to converge the non-differentiable method using relaxed Picard scheme.}
	\label{fig.nl-conv-p1}
\end{figure}

\begin{figure}[h]
	\centering
	\includegraphics[width=\textwidth]{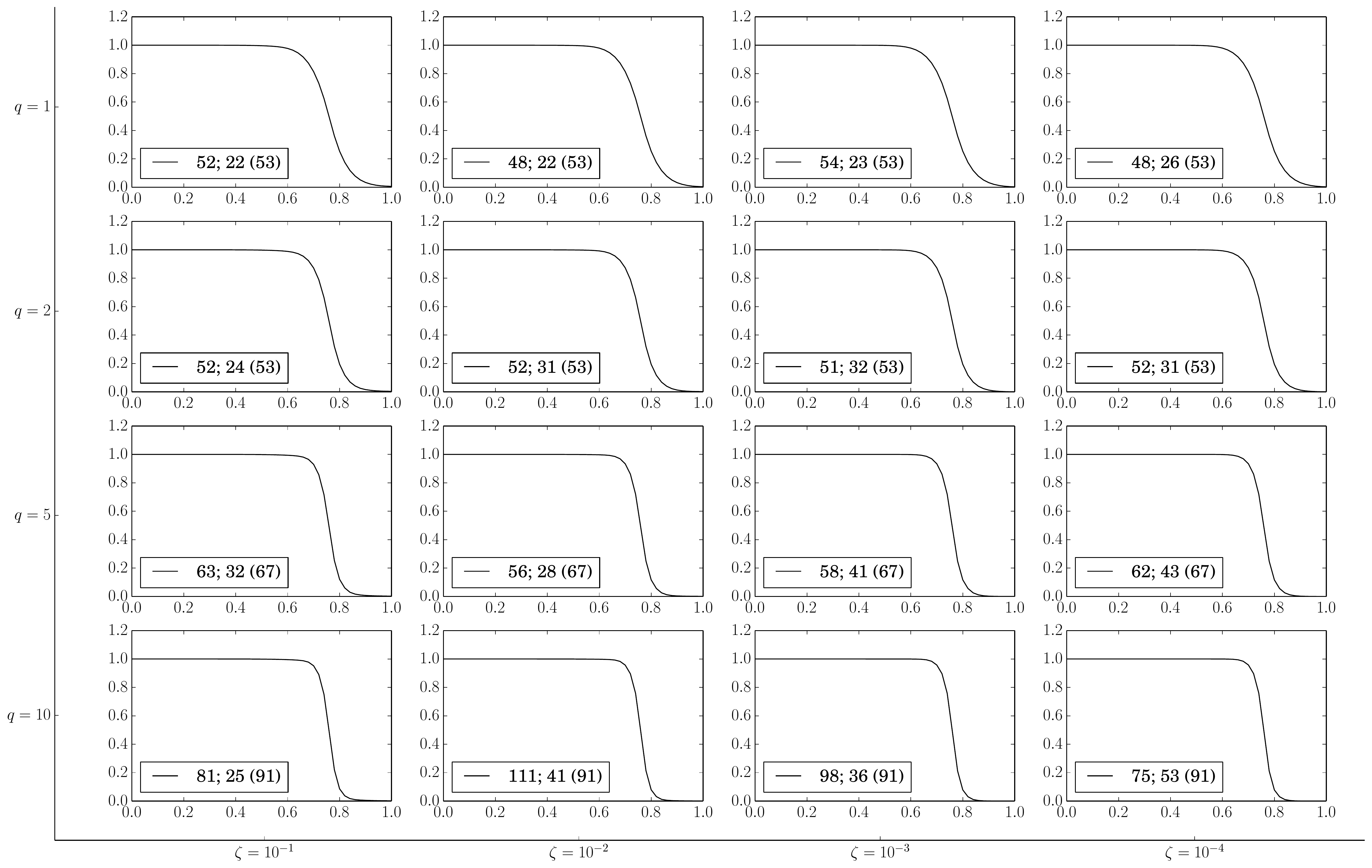}
	\caption{Effect of the regularization parameters for second order discretization. The numbers in legends are the number of nonlinear iterations performed. First number is for relaxed Picard and the next for hybrid scheme, both for the regularized stabilization. The number in brackets is the number of iterations required to converge the non-differentiable method using relaxed Picard scheme.}
	\label{fig.nl-conv-p2}
\end{figure}

\subsection{1D Sharp layer propagation} % 1D space--time == sharp layer (order comp.)
The performance of the stabilization schemes is analyzed as the discretization order is increased. To this end, we use again the previous problem \eqref{eq.sharplayer}. The regularization parameters are kept fixed, while the discretization is modified both in terms of the order of accuracy and the number of control points.

We use a nonlinear tolerance of $10^{-5}$. The regularization parameters used are $q=10$, $\varepsilon = 10^{-8}$, $\sigma=10^{-6}$, and $\gamma=10^{-10}$. With this setting, we solve the above problem using a discretization that keeps the number of control points fixed as the order is increased, and another one using the $k$-refinement defined in the previous experiment. For the former, we use a discretization of 120 by 60 control points. For the latter, we start with a first order discretization of 120 by 60 control points and refine as previously explained.

In Fig. \ref{sfig.1DconvIsotropicA} the solution at $t=0.5$ is shown for different orders and fixed number of control points, and using the $k$-refinement in Fig. \ref{sfig.1DconvIsotropicB}. We observe that for non-smooth solutions, fixing the number of control points and increasing the order does not improve the results. \correction{This is a consequence of the underlying discretization properties. The support of the shape functions becomes larger as the order is increased. Therefore, nonsmooth solutions become slightly more smeared}. In the case of Fig. \ref{sfig.1DconvIsotropicB}, as expected, we observe better approximations as the order is increased using the $k$-refinement. Hence, better results might be expected as the order is increased for problems that combine discontinuities and smooth profiles.

In Fig. \ref{fig.1DconvPartitioned}, similar results are shown when using the time integration scheme proposed in Sect. \ref{sec:time-integration}. A small degradation of the results can be seen in Fig. \ref{sfig.1DconvPartitionedA} as we increase the discretization order. In a similar trend, we observe less improvement in Fig. \ref{sfig.1DconvPartitionedB} than in Fig. \ref{sfig.1DconvIsotropicB}. We attribute this degradation to the time partitions, which becomes more evident as the subdomains are smaller. \correction{In particular, at the boundary of each partition the method might slightly increase the amount of diffusion introduced. At these boundaries, the shock detector relies on a smaller domain to determine if the \ac{DMP} is satisfied. Therefore, it is more likely to introduce more diffusion. On the other hand, the partition itself modifies the scheme introducing some error as shown in Sect.\ \ref{ssec.test1}.}

\begin{figure}[h]
	\centering
	\begin{subfigure}[b]{0.48\textwidth}
		\includegraphics[width=\textwidth]{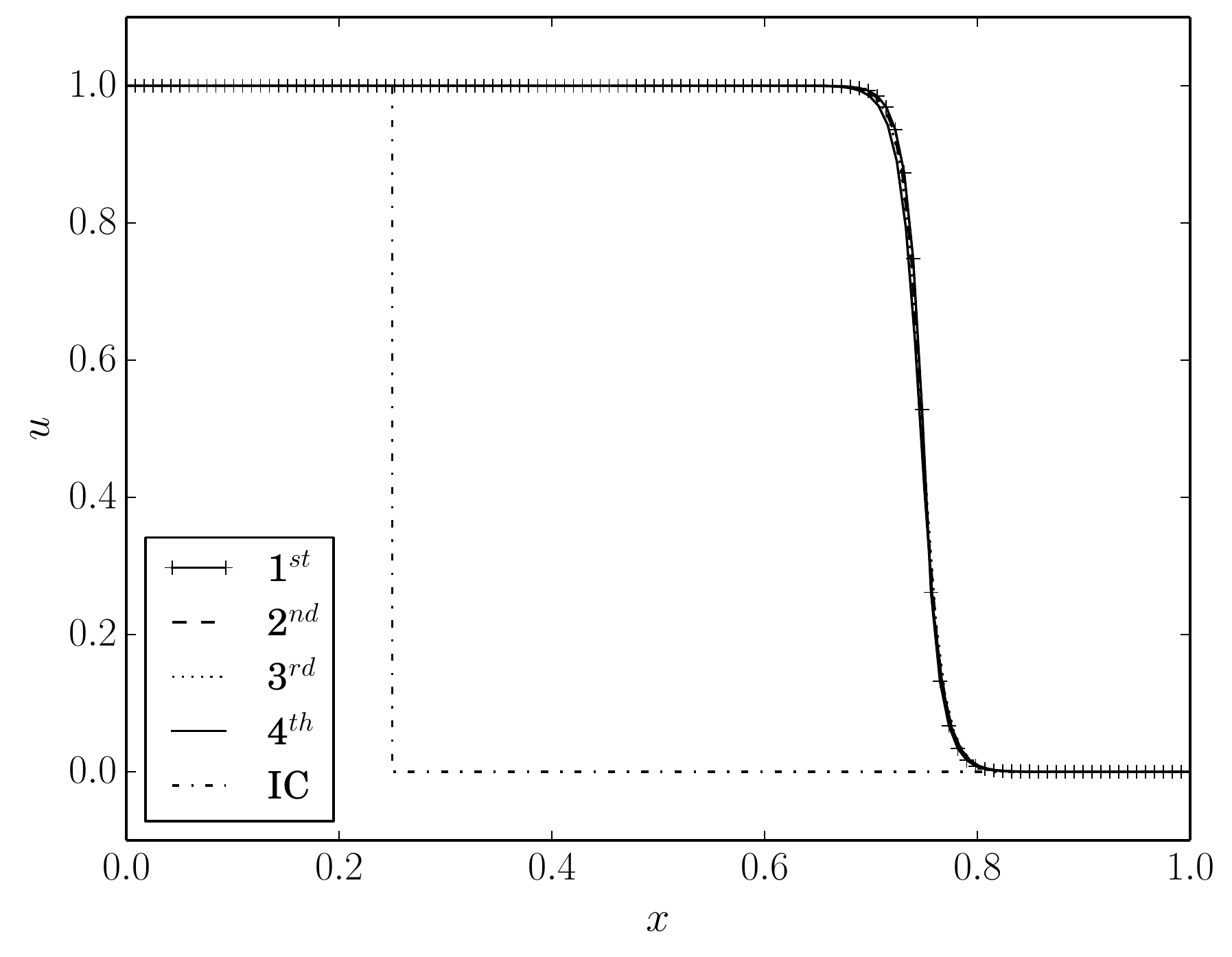}
		\caption{Solutions increasing the order while keeping fixed the number of control points.}
		\label{sfig.1DconvIsotropicA}
	\end{subfigure}
	\hfill
	\begin{subfigure}[b]{0.48\textwidth}
		\includegraphics[width=\textwidth]{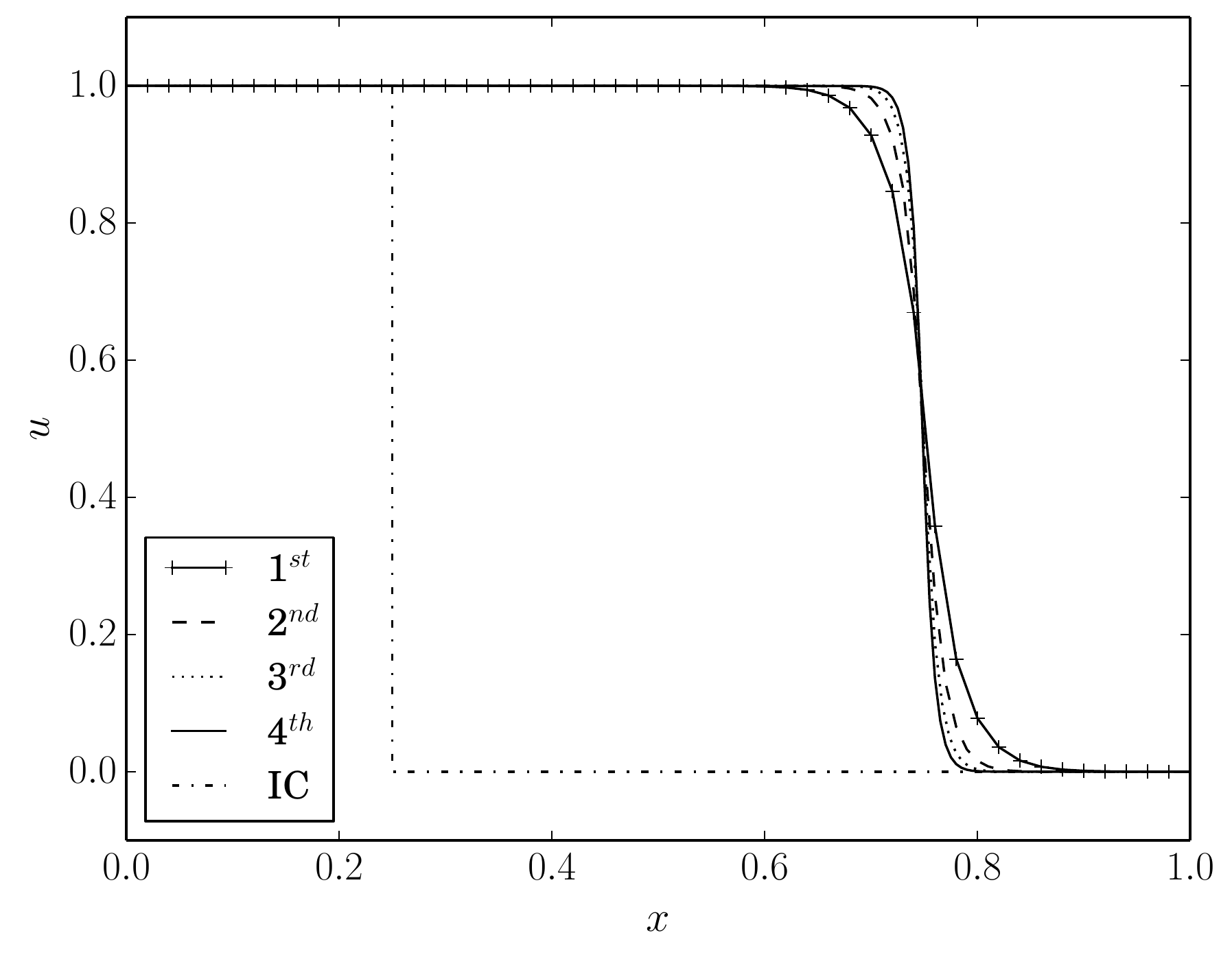}
		\caption{Solution increasing the order using the $k$-refinement process described above.}
		\label{sfig.1DconvIsotropicB}
	\end{subfigure}
\caption{Solution of problem \eqref{eq.sharplayer} at $t=0.5$ for first to fourth order discretizations.}
	\label{fig.1DconvIsotropic}
\end{figure}

\begin{figure}[h]
	\centering
	\begin{subfigure}[b]{0.48\textwidth}
		\includegraphics[width=\textwidth]{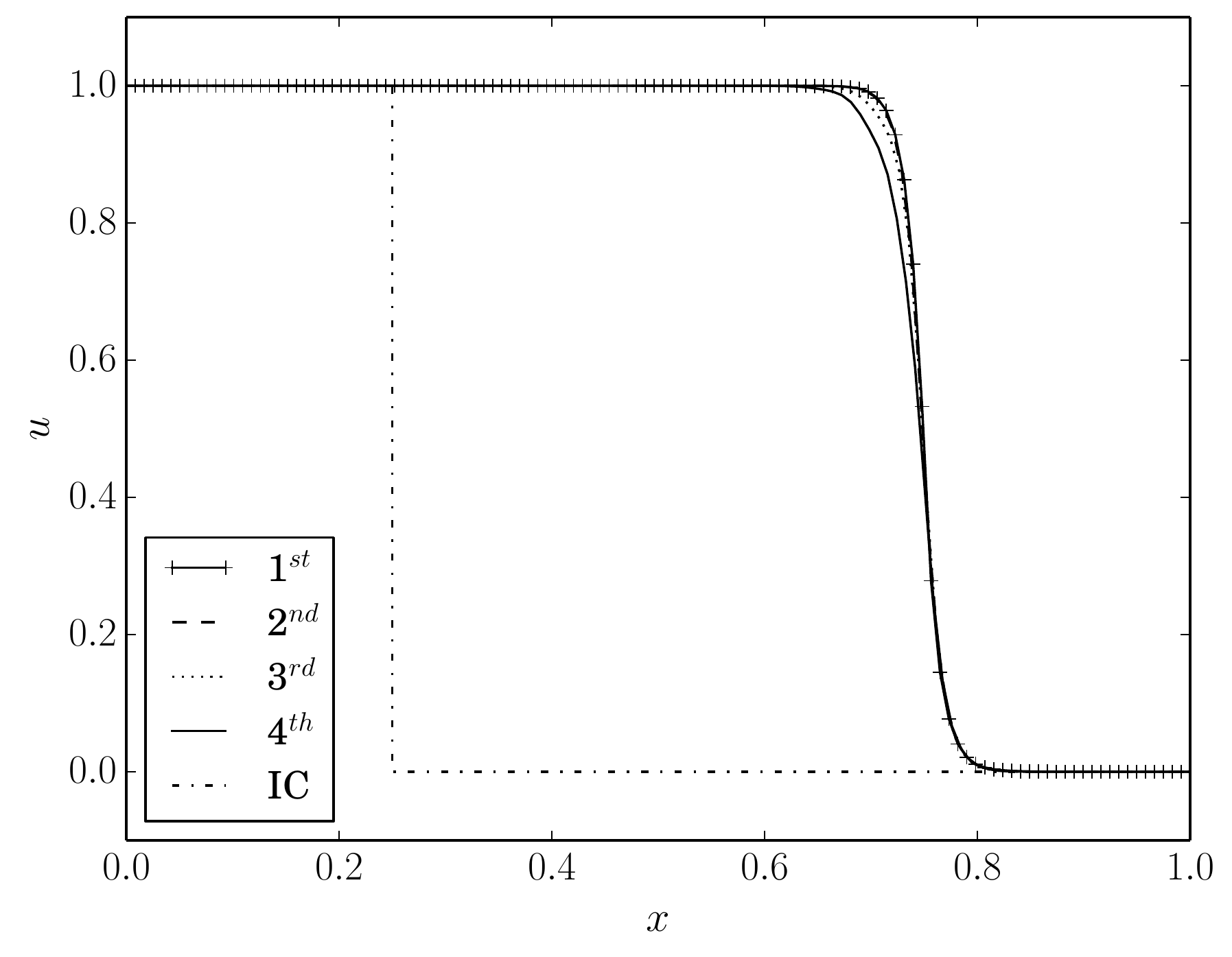}
		\caption{Solution using fixed number of control points, and time integration defined in Sect. \ref{sec:time-integration} with 5 partitions.}
		\label{sfig.1DconvPartitionedA}
	\end{subfigure}
	\hfill
	\begin{subfigure}[b]{0.48\textwidth}
		\includegraphics[width=\textwidth]{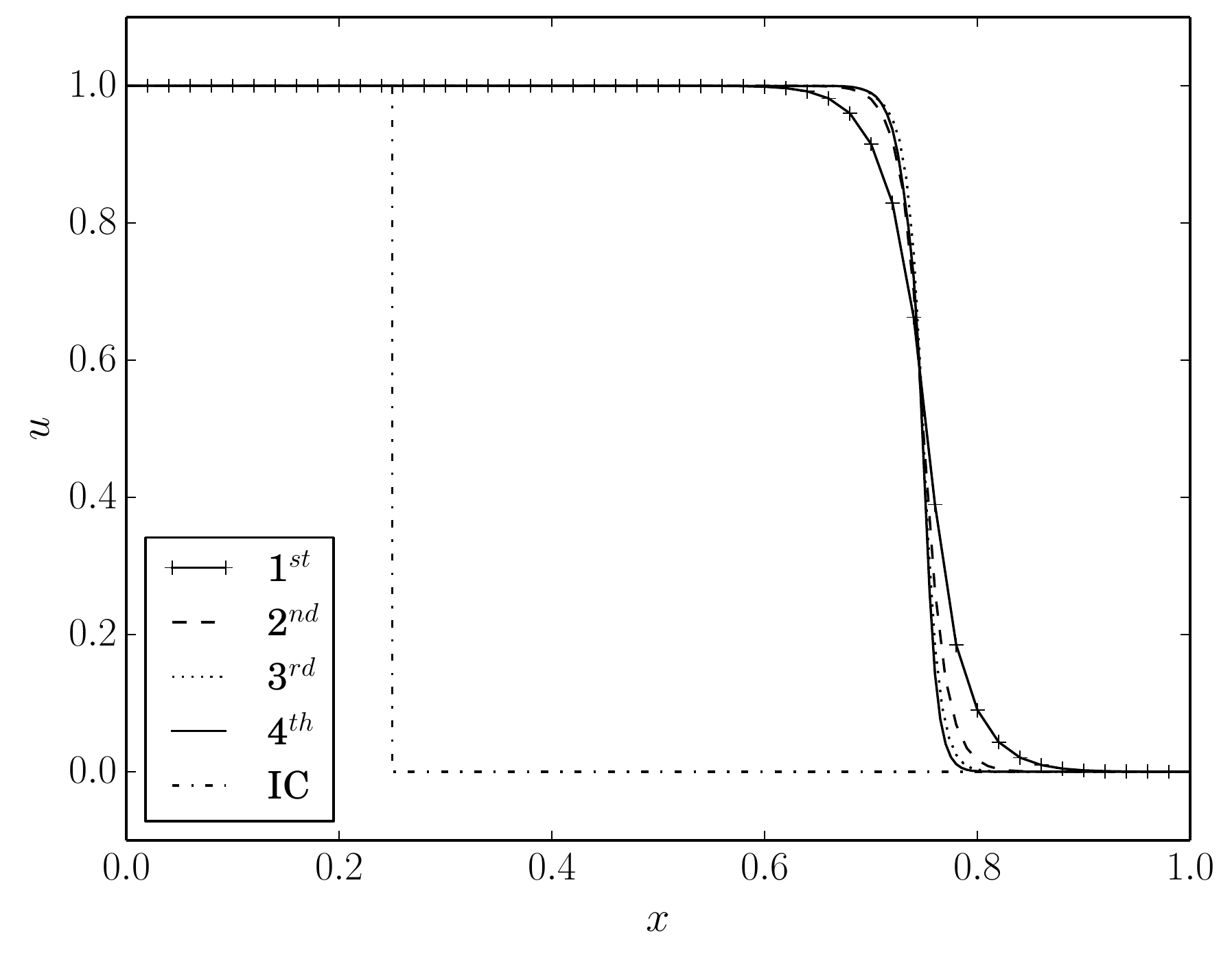}
		\caption{Solution using $k$-refinement, and time integration defined in Sect. \ref{sec:time-integration} with 5 partitions.}
		\label{sfig.1DconvPartitionedB}
	\end{subfigure}
	\caption{Solution of problem \eqref{eq.sharplayer} at $t=0.5$ for first to fourth order discretizations.}
	\label{fig.1DconvPartitioned}
\end{figure}

\subsection{Boundary layer} % 2D steady C-D problem
In this section the effect of the discretization order in a convection-diffusion problem is analyzed. To this end, we solve a problem with the propagation of a sharp layer and a boundary layer. In particular, we solve for $\domain\doteq[0,1]^2$
\begin{equation}\label{eq.layer} 
\left\{\begin{array}{rl}
-10^{-4}\Delta u + \conv\cdot\gradient u= 0 & \text{in } \domain \\
u = u_D & \text{at } \partial\domain
\end{array}\right.,
\end{equation}
where $\bsbeta=(\cos \theta,\sin \theta)$, $\theta=-\pi/3$, and the boundary conditions are defined as 
\begin{equation}
u_D = \left\{\begin{array}{cc}
\half+\frac{1}{\pi}\arctan\left(10^{-4}  \left(y-5/6\right)\right) & \text{if } y = 0 \\
0 & \text{otherwise}
\end{array}\right. .
\end{equation}
   
For this test, we use the following settings: a nonlinear tolerance of $10^{-8}$, $q=2$, $\varepsilon=10^{-8}$, $\sigma=10^{-6}$, and $\gamma=10^{-10}$. In Fig. \ref{fig.layer-3d}, the solution for $p=4$ is depicted. The converged solution does not exhibit any oscillation. Very sharp layers are obtained for this parameter setting. In Fig. \ref{fig.layer-cuts}, we show the profile of the solution at $y=0.1$ for different orders. In this case, we start with a discretization of 50 control points per direction. Then, we increase the order using the $k$-refinement used previously. As previously observed for transient problems, we observe an improvement of the solution as the order is increased.

\begin{figure}[h]
	\centering
	\begin{subfigure}[b]{0.48\textwidth}
		\includegraphics[width=0.90\textwidth]{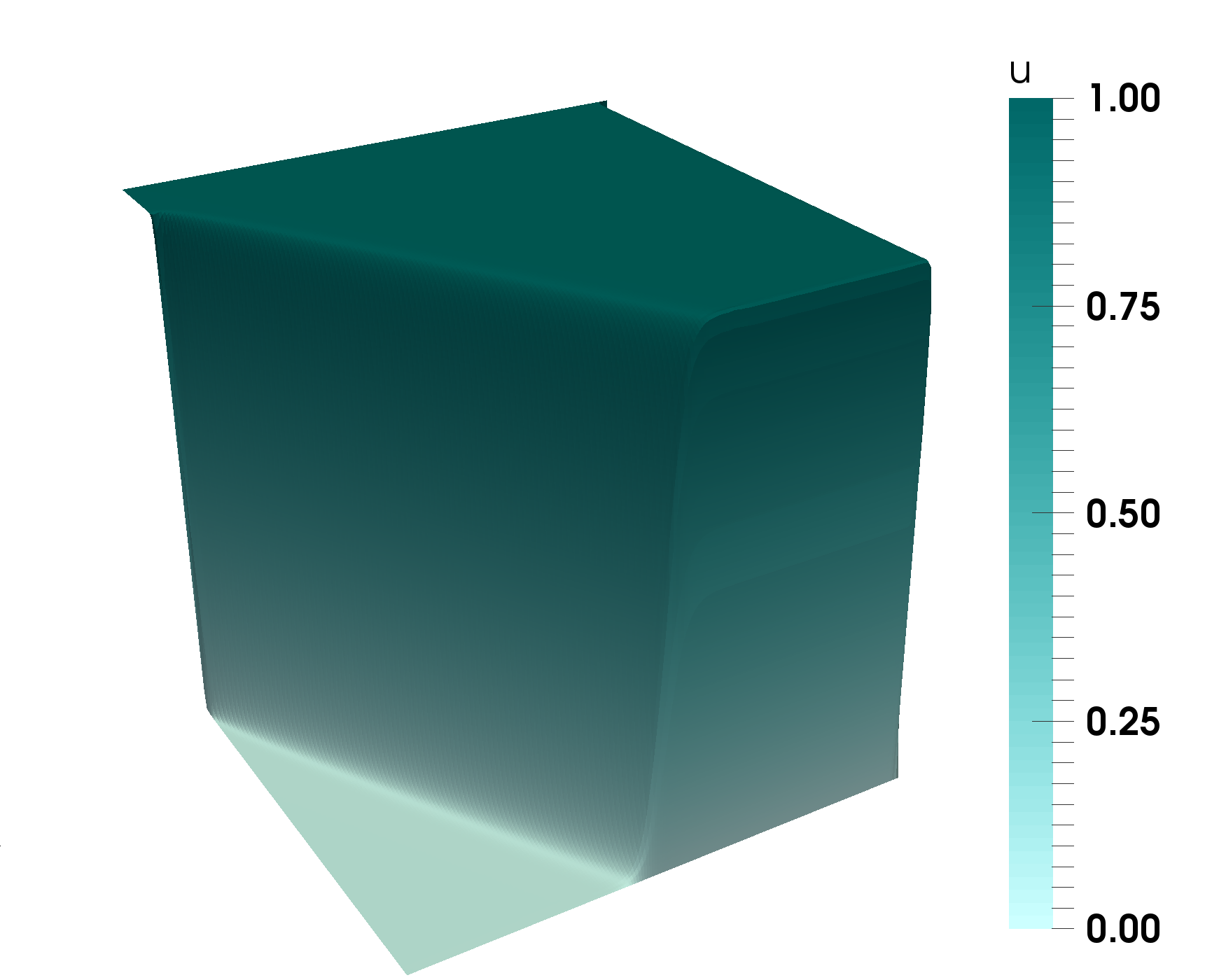}
		\vspace{6mm}
		\caption{3D representation of the solution.}
		\label{fig.layer-3d}
	\end{subfigure}
	\begin{subfigure}[b]{0.48\textwidth}
		\includegraphics[width=\textwidth]{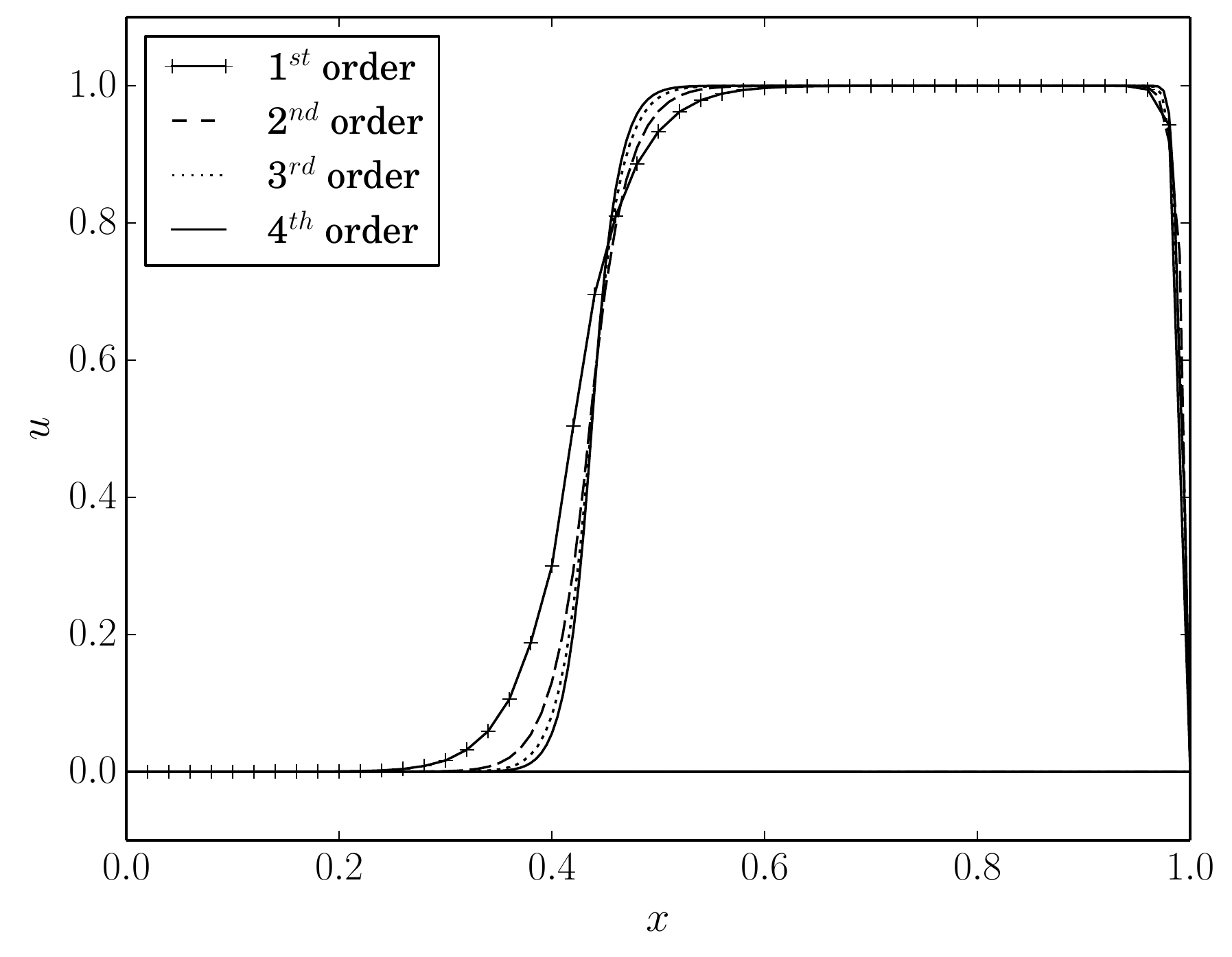}
		\caption{Profiles at $y=0.1$.}
		\label{fig.layer-cuts}
	\end{subfigure}
	\caption{Solution of problem \eqref{eq.layer} using scheme \eqref{eq.stabilized-problem}, and different discretization orders.}
\end{figure}

\subsection{Three Body rotation} % 2D transient pure convection problem
Finally, we solve the transient pure convection problem \eqref{eq.sharplayer} in $\domain\times(0,1]$ for $\domain=[0,1]^2$, with $\conv=(-2\pi(y-0.5),2\pi(x-0.5))$. Initial conditions are given in \cite{Kuzmin2010}. Its interpolation in a first order $200\times 200$ control point mesh is depicted in Fig. \ref{fig.3body-ic}. The analytical solution of this problem is simply the translation of the profiles in the direction of the convection. In particular, for $t=1$, one revolution is completed and the solution is equal to the initial conditions. The purpose of this test is to evaluate how diffusive is the proposed scheme. We perform this evaluation evolving the solution until $t=1$ and comparing the results with the initial conditions.

The solution is computed using scheme \eqref{eq.stabilized-problem} in combination with the shock detector in \eqref{eq.smthdetector}. We use the following parameters for the stabilization: $q=10$, $\sigma=10^{-6}$, $\varepsilon=10^{-8}$, and $\gamma=10^{-10}$. % tolerance of $10^{-3}$ and maximum of nonlinear iterations of 10. 
Different meshes, time partitions, and discretization orders are used in this experiment. We start with a linear discretization of $100\times 100$ control points in space, and 500 in time divided in 125 subdomains. Then, we increase the discretization order to $p=2$ using the $k$-refinement. In order to compare first and second order discretizations, but using a similar number of control points we use a discretization with $200\times 200$ control points in space, and 1000 in time divided in 250 subdomains.
Finally, we assess the effect of the partitions in the temporal direction. We compare the previous discretization of $100\times 100\times 500$ control points divided in 125 subdomains, with the same discretization divided in 250 subdomains. We do the same comparison for the second order discretization using 125 subdomains and when it is divided in 250 subdomains.

\begin{figure}[h]
	\centering
	\begin{subfigure}[t]{0.48\textwidth}
		\includegraphics[width=\textwidth]{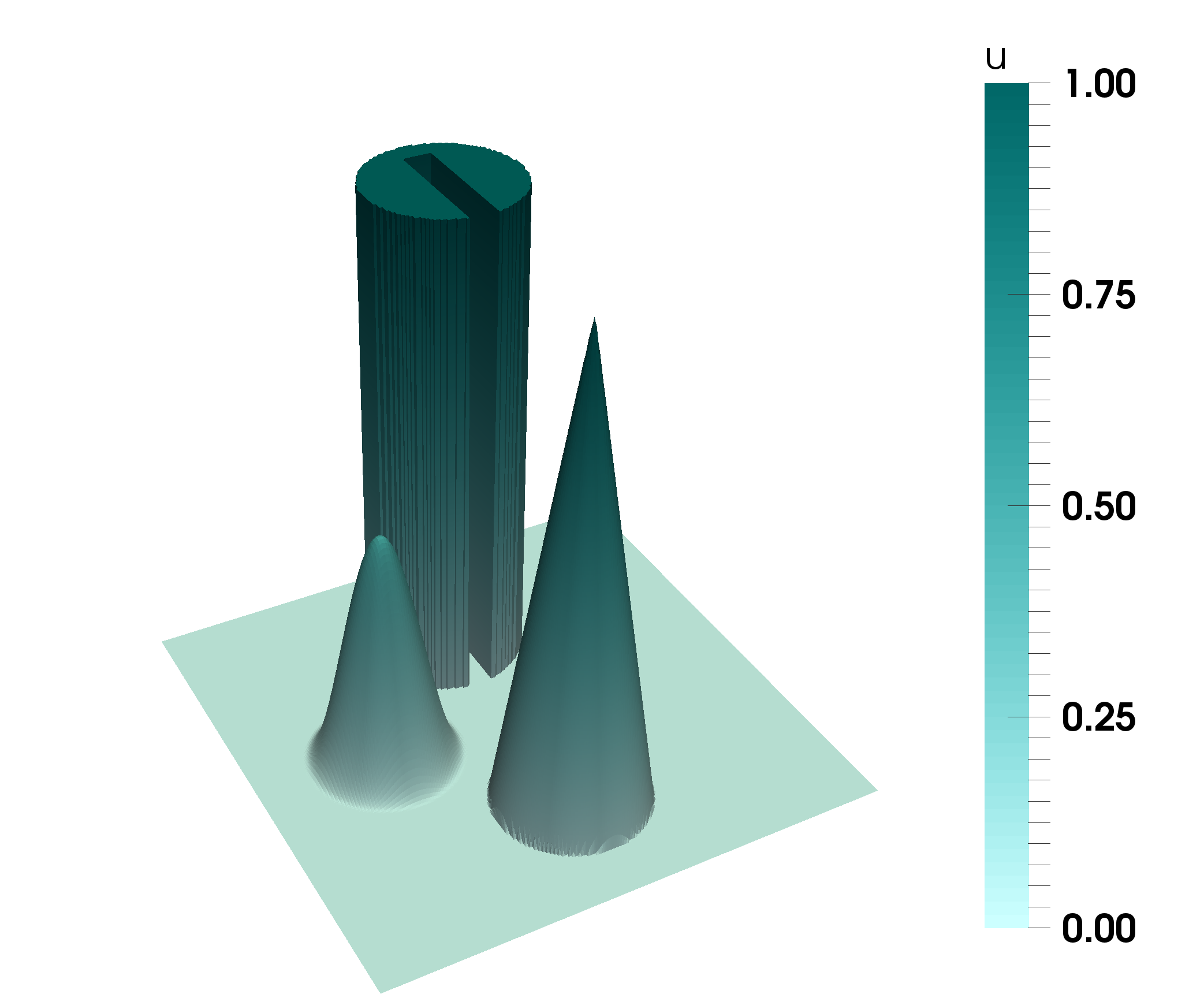}
		\caption{Initial conditions of the 3 body rotation.}
		\label{fig.3body-ic}
	\end{subfigure}
	\begin{subfigure}[t]{0.48\textwidth}
		\includegraphics[width=\textwidth]{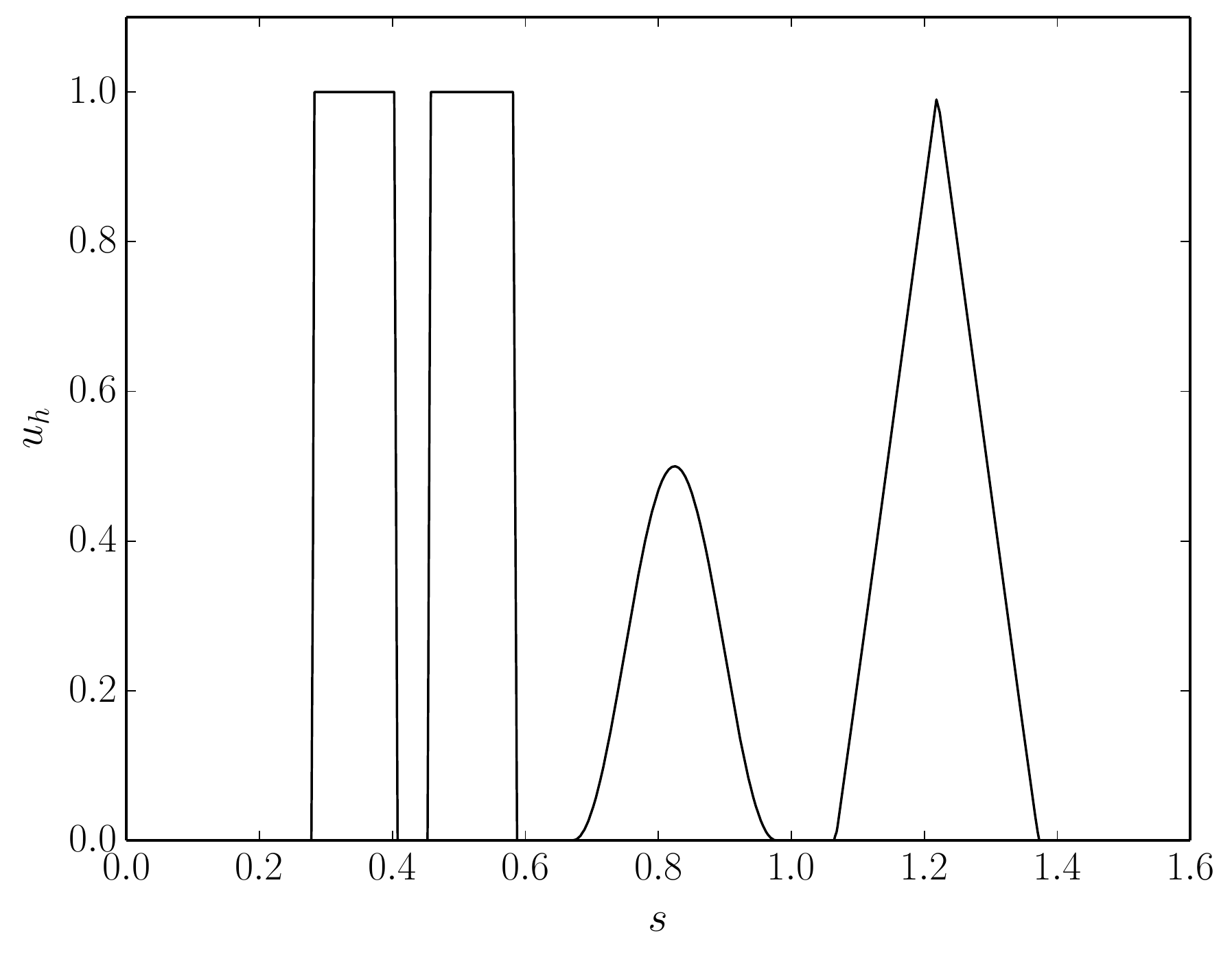}
		\caption{Profile at $\sqrt{(x-0.5)^2+(y-0.5)^2} = 0.25$.}
		\label{fig.3body-ics}
	\end{subfigure}
	\caption{Three body rotation test initial conditions. }
	\label{fig.3body-icfigs}
\end{figure}

Fig. \ref{fig.3body-4} shows the solutions for $100\times 100$ meshes, and 125 subdomains in time, whereas Fig. \ref{fig.3body-2} shows the ones for 250 subdomains. In both cases, a great improvement can be observed as we increase the discretization order. However, the computational cost is also increased. It is interesting to compare the solutions for first and second order discretizations using meshes with similar amount of control points, namely solutions at Fig. \ref{fig.3body-200-figs} and \ref{fig.3body-p2-4}. For this particular problem, using a higher order discretization with similar number of control points does not improve the solution, which it is actually slightly more diffusive for $p=2$. It is also worth mentioning that increasing the discretization order does not modify the behavior of the solution in terms of clipping or terracing. Comparing Figs. \ref{fig.3body-p1-2} and \ref{fig.3body-p1-4}, we observe that the scheme becomes more dissipative as the number of partitions is increased. This is even clearer in Fig. \ref{fig.3body-slices}, where the profile of the solution at $s\doteq\{(x,y) : \sqrt{(x-0.5)^2+(y-0.5)^2} = 0.25\}$ is depicted.

\begin{figure}[h]
	\centering
	\begin{subfigure}[t]{0.48\textwidth}
		\includegraphics[width=\textwidth]{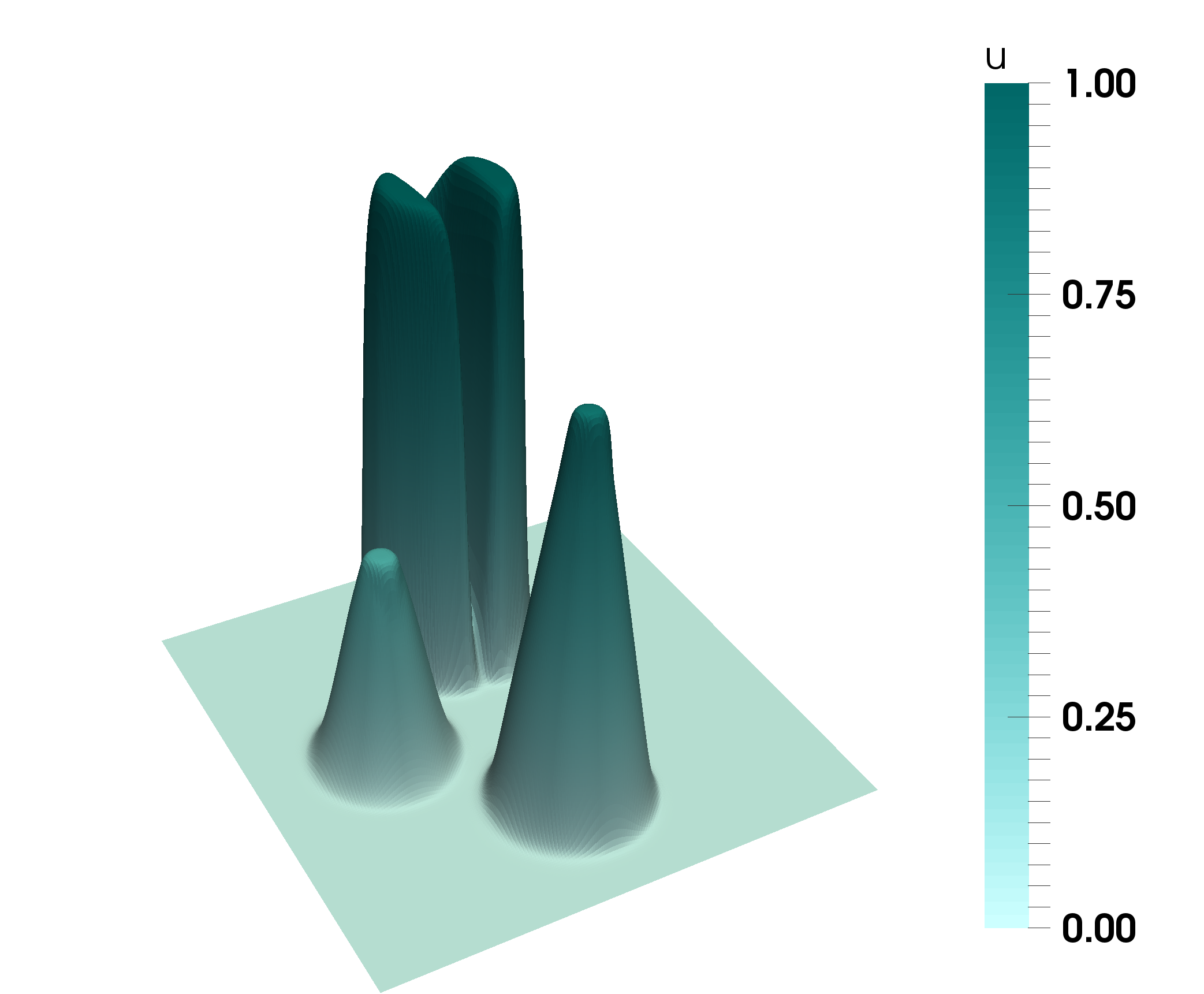}
		\caption{3D view.}
		\label{fig.3body-200}
	\end{subfigure}
	\begin{subfigure}[t]{0.48\textwidth}
		\includegraphics[width=\textwidth]{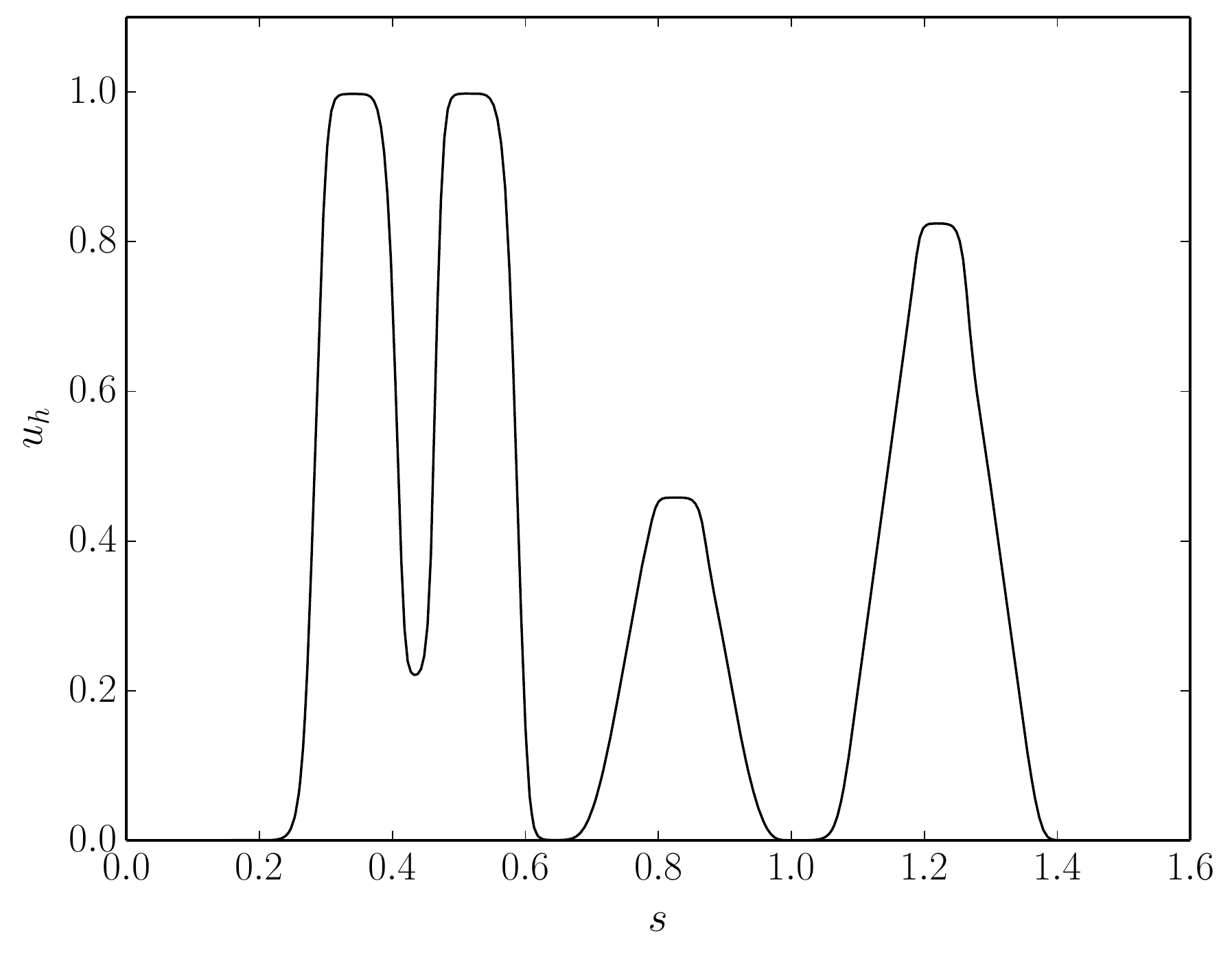}
		\caption{Profile at $\sqrt{(x-0.5)^2+(y-0.5)^2} = 0.25$.}
		\label{fig.3body-200s}
	\end{subfigure}
	\caption{Three body rotation test results at $t=1$ using scheme \eqref{eq.stabilized-problem}, $q=10$, $\sigma=10^{-6}$, $\varepsilon=10^{-8}$, and $\gamma=10^{-10}$. A first order discretization of $200\times 200\times 1000$ control points is used with 250 subdomains in the temporal direction.}
	\label{fig.3body-200-figs}
\end{figure}

\begin{figure}[h]
	\centering
	\begin{subfigure}[t]{0.48\textwidth}
		\includegraphics[width=\textwidth]{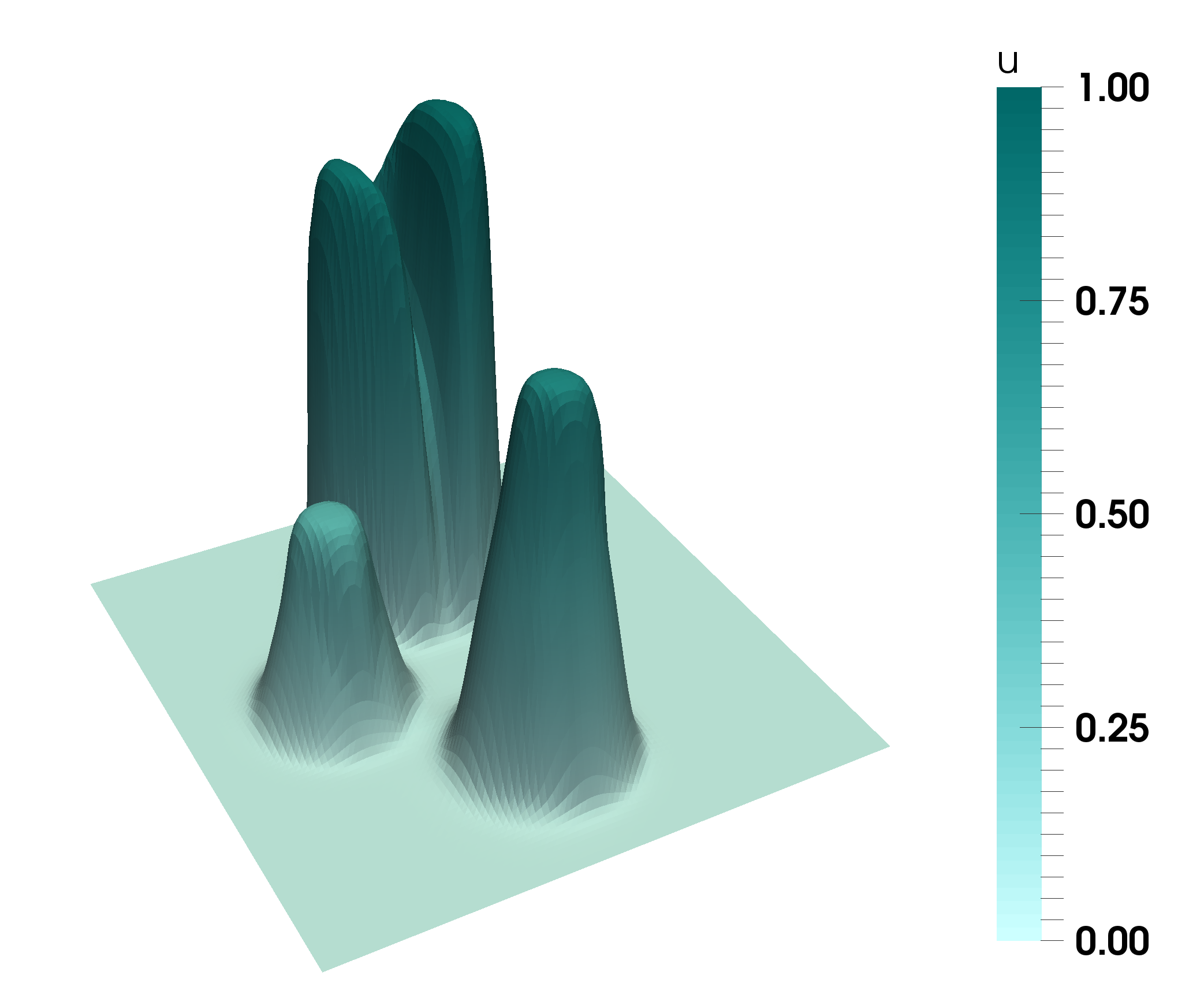}
		\caption{Solution for a first order discretization.}
		\label{fig.3body-p1-4}
	\end{subfigure}
	\begin{subfigure}[t]{0.48\textwidth}
		\includegraphics[width=\textwidth]{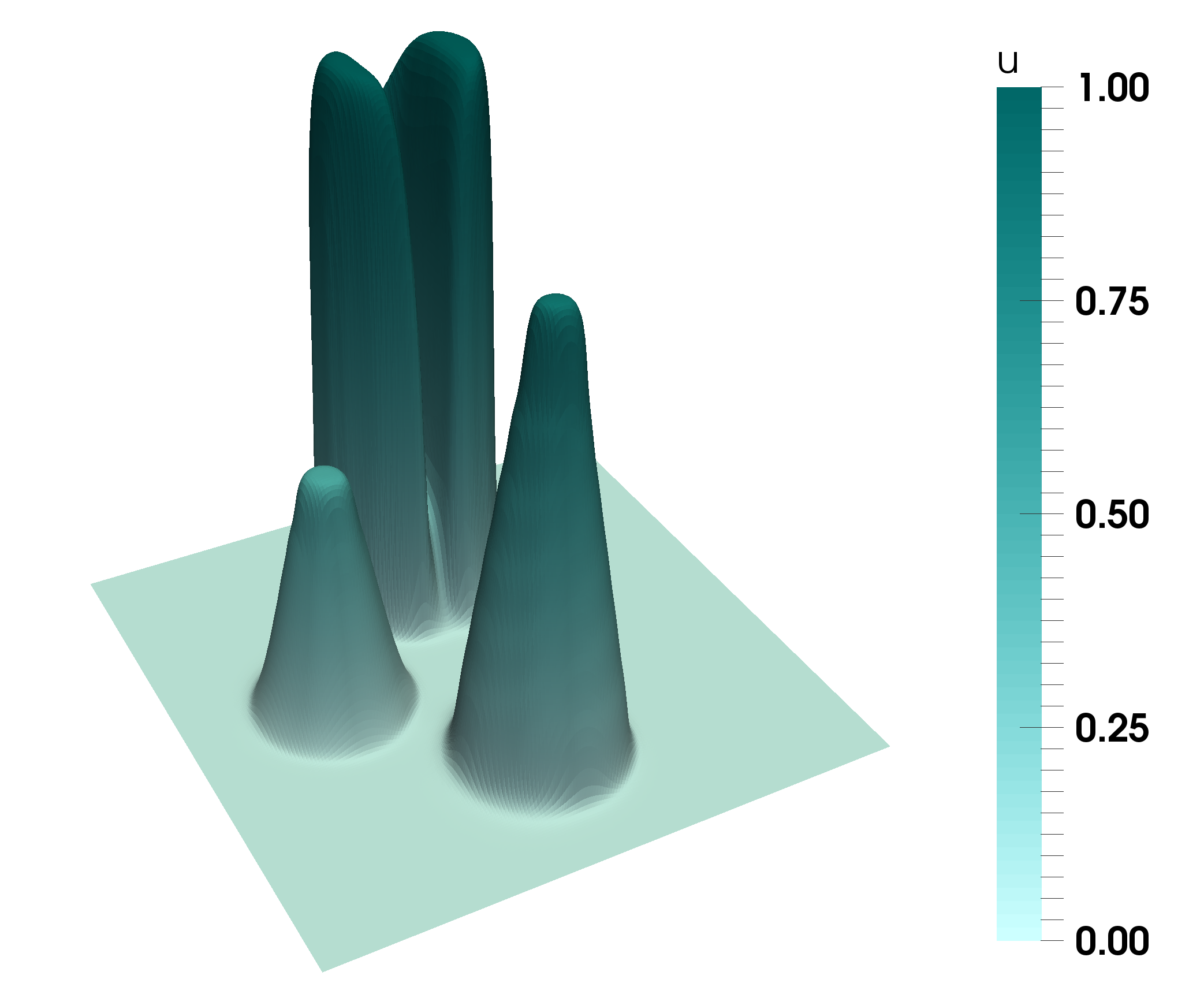}
		\caption{Solution after one $k$-refinement.}
		\label{fig.3body-p2-4}
	\end{subfigure}
	\caption{Three body rotation test results at $t=1$ using scheme \eqref{eq.stabilized-problem}, $q=10$, $\sigma=10^{-6}$, $\varepsilon=10^{-8}$, and $\gamma=10^{-10}$. A first order discretization of $100\times 100\times 500$ control points is used. The second order discretization is obtained using $k$-refinement. 125 subdomains in the temporal direction have been used.}
	\label{fig.3body-4}
\end{figure}

\begin{figure}[h]
	\centering
	\begin{subfigure}[t]{0.48\textwidth}
		\includegraphics[width=\textwidth]{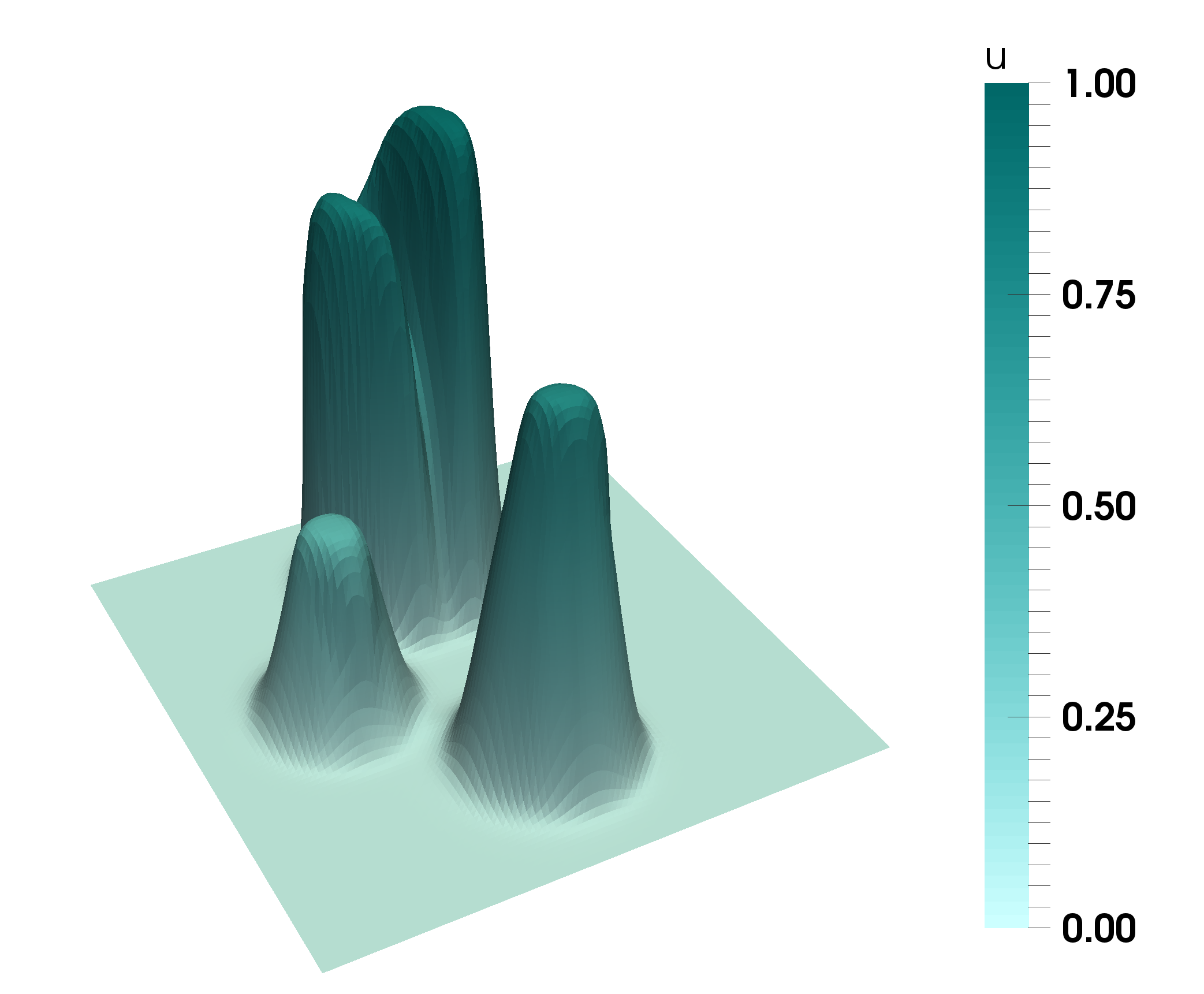}
		\caption{Solution for first order discretization.}
		\label{fig.3body-p1-2}
	\end{subfigure}
	\begin{subfigure}[t]{0.48\textwidth}
		\includegraphics[width=\textwidth]{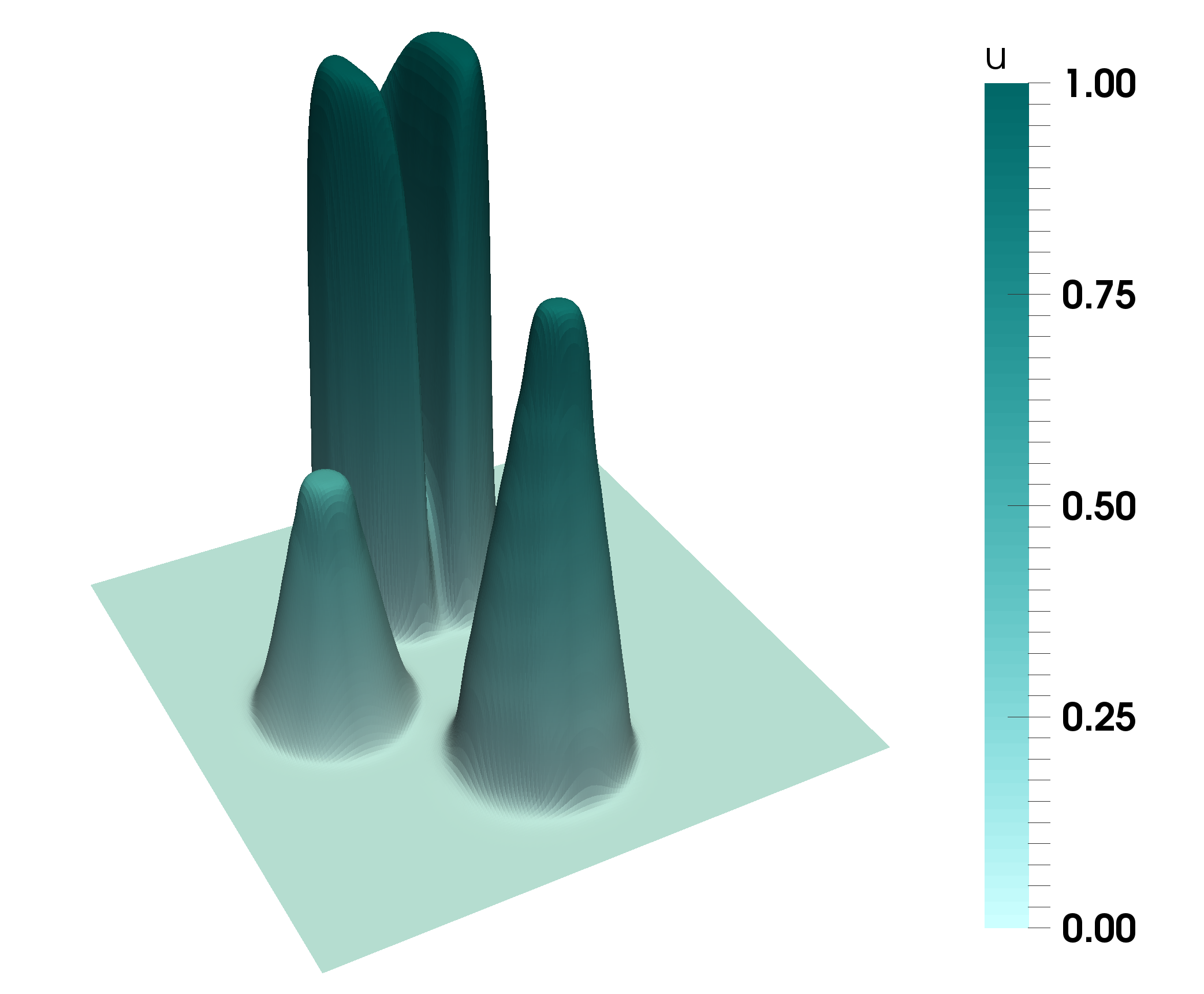}
		\caption{Solution after one $k$-refinement.}
		\label{fig.3body-p2-2}
	\end{subfigure}
	\caption{Three body rotation test results at $t=1$ using scheme \eqref{eq.stabilized-problem}, $q=10$, $\sigma=10^{-6}$, $\varepsilon=10^{-8}$, and $\gamma=10^{-10}$. A first order discretization of $100\times 100\times 500$ control points is used. The second order discretization is obtained using $k$-refinement. 250 subdomains in the temporal direction have been used.}
	\label{fig.3body-2}
\end{figure}

\begin{figure}[h]
	\centering
	\begin{subfigure}[t]{0.48\textwidth}
		\includegraphics[width=\textwidth]{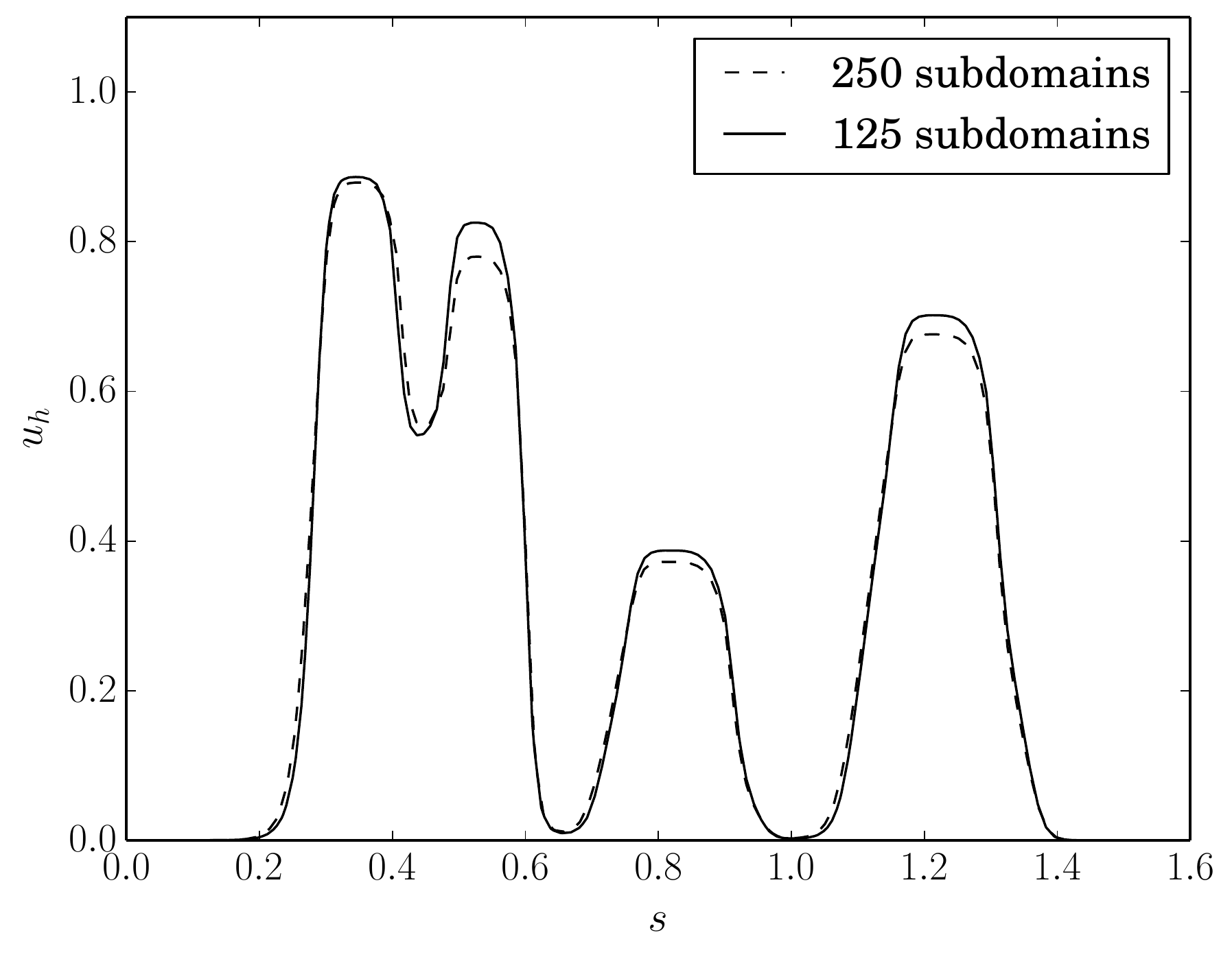}
		\caption{Solution for first order discretization.}
		\label{fig.3body-slice-p1}
	\end{subfigure}
	\begin{subfigure}[t]{0.48\textwidth}
		\includegraphics[width=\textwidth]{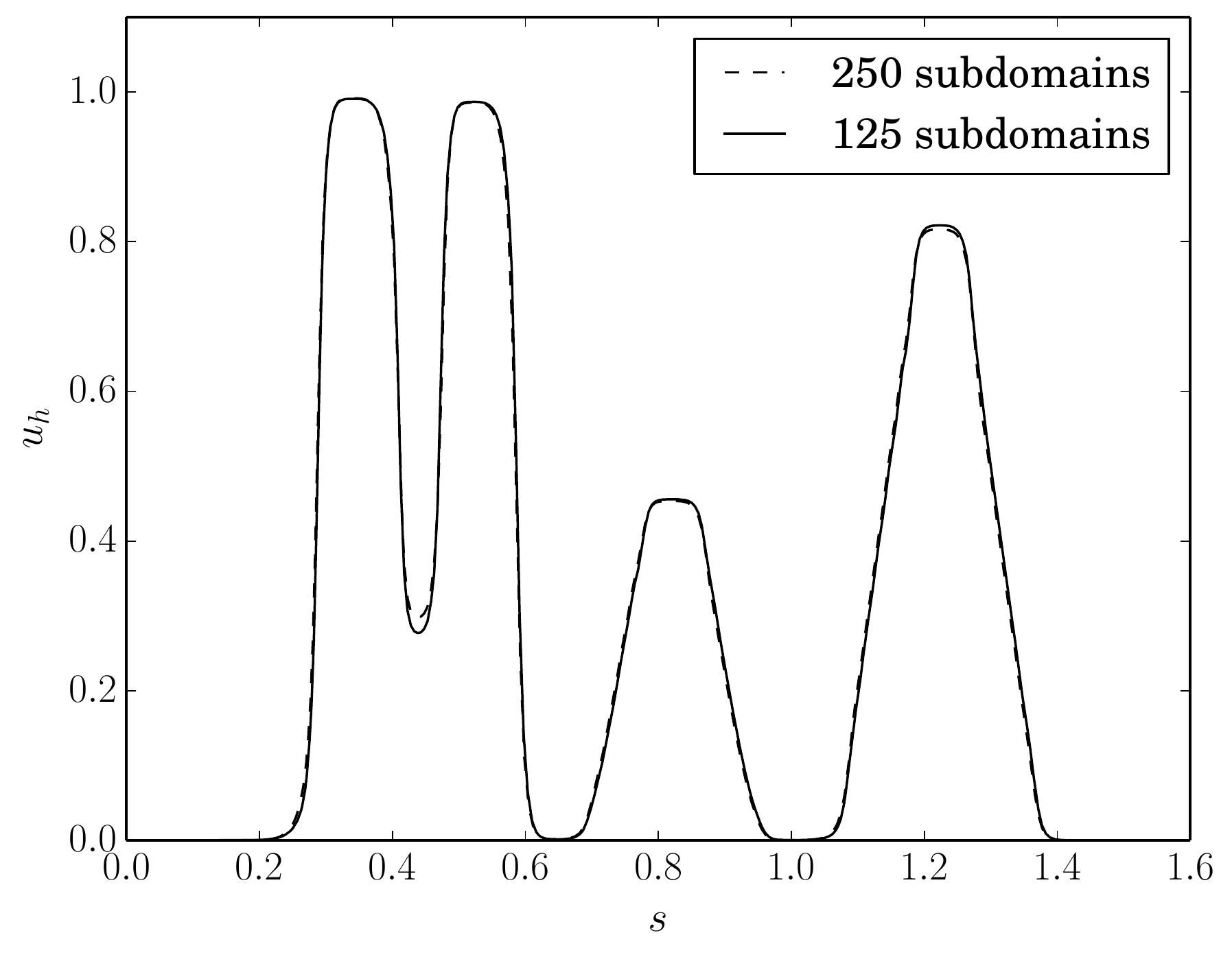}
		\caption{Solution after one $k$-refinement.}
		\label{fig.3body-slice-p2}
	\end{subfigure}
	\caption{Three body rotation test profiles for $t=1$ at $\sqrt{(x-0.5)^2+(y-0.5)^2} = 0.25$ using scheme \eqref{eq.stabilized-problem}, $q=10$, $\sigma=10^{-6}$, $\varepsilon=10^{-8}$, and $\gamma=10^{-10}$. A first order discretization of $100\times 100\times 500$ control points is used. The second order discretization is obtained using $k$-refinement. 125 and 250 subdomains in the temporal direction have been used.}
	\label{fig.3body-slices}
\end{figure}

\section{Conclusions}\label{sec.conclusions}

In the present work, an extension of \cite{badia_monotonicity-preserving_2017} to isogeometric analysis methods have been developed. The proposed method is unconditionally \ac{DMP} preserving for arbitrary high-order discretizations in space and time without any CFL-like condition. Furthermore, it is shown to be linearity-preserving in a space--time sense. Moreover, the regularized version is shown to yield better convergence behavior, specially when for the hybrid Picard--Newton method.

Moreover, the numerical experiments show that increasing the discretization order yield much better solutions. However, as the order is increased the number of control points and the computational cost is also increased. On the contrary, if the order is increased while the number of control points is fixed, then similar or even slightly more diffusive results are obtained for non-smooth solutions. Hence, for problems with regions of smooth and non-smooth solutions a high-order is expected to outperform linear discretizations with similar amount of control points. Furthermore, a method capable of providing solutions that satisfy the \ac{DMP} for high-order discretizations is of special interest in $hp$-adaptive schemes, since the usage of first order discretizations in shocks is not required.

In addition, a partitioned scheme that does not harm any monotonicity property is presented. This scheme reduces significantly the computational cost of the original space--time scheme. It is important to mention, that this partitioning slightly increases the error. However, this approach allows finer meshes, and thus, in practice better solutions can be obtained.

\section*{Acknowledgments}
J. Bonilla gratefully acknowledges the support received from ”la Caixa” Foundation through
its PhD scholarship program (LCF/BQ/DE15/10360010). S. Badia gratefully acknowledges the support received 
from the Catalan Government through the ICREA Acad\`emia Research Program. We acknowledge the financial support to CIMNE via the CERCA Programme  / Generalitat de Catalunya.

\bibliographystyle{siam}
\bibliography{refs}

\end{document}